\documentclass[12pt,reqno]{amsart}
\usepackage{amssymb,color}

\newcommand{\e}{\varepsilon}
\newcommand{\om}{\omega}
\newcommand{\LL}{\mathcal{L}}
\newcommand{\LH}{\mathfrak{L}}
\newcommand{\M}{\mathcal{M}}
\newcommand{\la}{\lambda}
\newcommand{\al}{\alpha}

\newcommand{\fy}{\varphi}

\newcommand{\p}{\partial}
\newcommand{\I}{\infty}

\newcommand{\ti}{\widetilde}
\newcommand{\R}{\mathbb{R}}
\newcommand{\C}{\mathbb{C}}
\newcommand{\N}{\mathbb{N}}
\newcommand{\Z}{\mathbb{Z}}
\newcommand{\F}{\mathcal{F}}
\renewcommand{\S}{\mathcal{S}}
\newcommand{\T}{\mathcal{T}}
\renewcommand{\Re}{\mathop{\mathrm{Re}}}
\renewcommand{\Im}{\mathop{\mathrm{Im}}}
\renewcommand{\bar}{\overline}
\renewcommand{\hat}{\widehat}

\newcommand{\ro}{\rho}

\numberwithin{equation}{section}

\newtheorem{thm}{Theorem}[section]
\newtheorem{cor}[thm]{Corollary}
\newtheorem{lem}[thm]{Lemma}
\newtheorem{prop}[thm]{Proposition}

\theoremstyle{remark}

\newcommand{\ran}{\rangle}
\newcommand{\lan}{\langle}

\newcommand{\weak}[1]{{\text{w-}#1}}
\newcommand{\lec}{\lesssim}
\newcommand{\gec}{\gtrsim}

\newcommand{\EQ}[1]{\begin{equation} \begin{split} #1 \end{split} \end{equation}}
\newcommand{\Del}[1]{}
\newcommand{\CAS}[1]{\begin{cases} #1 \end{cases}}
\newcommand{\pt}{&}
\newcommand{\pr}{\\ &}
\newcommand{\pq}{\quad}
\newcommand{\pn}{}

\newcommand{\prQQ}{\\ &\qquad\qquad}

\newcommand{\LR}[1]{{\lan #1 \ran}}
\newcommand{\de}{\delta}
\newcommand{\si}{\sigma}

\renewcommand{\t}{\tau}
\newcommand{\be}{\beta}
\newcommand{\ka}{\kappa}
\newcommand{\ga}{\gamma}
\newcommand{\x}{\xi}
\newcommand{\y}{\eta}
\newcommand{\z}{\zeta}
\newcommand{\s}{\sigma}

\newcommand{\na}{\nabla}

\renewcommand{\th}{\theta}

\newcommand{\supp}{\operatorname{supp}}

\newcommand{\D}{\mathcal{D}}

\newcommand{\B}{\mathcal{B}}
\newcommand{\HH}{\mathcal{H}}
\newcommand{\HL}{\mathfrak{H}}
\newcommand{\De}{\Delta}
\newcommand{\Om}{\Omega}

\newcommand{\IN}[1]{\text{ in }#1}

\newcommand{\mat}[1]{\begin{pmatrix} #1 \end{pmatrix}}

\newcommand{\Cu}{\bigcup}

\newcommand{\sg}{\mathfrak{s}}
\newcommand{\Sg}{\mathfrak{S}}
\newcommand{\sign}{\operatorname{sign}}
\newcommand{\WV}{\overrightarrow}
\newcommand{\K}{\mathcal{K}}
\newcommand{\uu}{w}
\newcommand{\VQ}{\mathfrak{Q}}
\newcommand{\A}{\mathcal{A}}
\newcommand{\nx}[1]{\ti{#1}}
\newcommand{\diff}[1]{{\triangleleft{#1}}}
\def\PS{\mathcal{PS}}
\def\calM{\mathcal{M}}
\newcommand{\Ex}{E_{ext}}
\def\spec{\mathrm{spec}}

\begin{document}

\author{K.~Nakanishi}
\address{Department of Mathematics, Kyoto University\\ Kyoto 606-8502, Japan}
\email{n-kenji@math.kyoto-u.ac.jp}

\author{W.~Schlag}
\address{Department of  Mathematics, The University of Chicago\\ Chicago, IL 60615, U.S.A.} 
\email{schlag@math.uchicago.edu} 

\thanks{The second author was supported in part by the National Science Foundation,  DMS-0617854  as well as by a Guggenheim fellowship.}

\title[Global dynamics for NLKG without radial assumption]{Global dynamics above the ground state \\ for the nonlinear Klein-Gordon equation \\ without a radial assumption}

\subjclass[2010]{35L70, 35Q55} 
\keywords{nonlinear wave equation, ground state, hyperbolic dynamics, stable manifold, unstable manifold, scattering theory, blow up}

\begin{abstract}
We extend our result \cite{NLKGrad} to the non-radial case, giving a complete classification of global dynamics of all solutions with energy at most slightly above that of the ground state for the nonlinear Klein-Gordon equation with the focusing cubic nonlinearity in three space dimensions.
\end{abstract}

\maketitle


\section{Introduction}
\label{sec:intro}

The nonlinear Klein-Gordon equation
\EQ{
\label{eq:NLKG3}
\Box u + u = \pm u^{3}, \quad (t,x)\in \R^{1+3}_{t,x}
}
has been considered by many authors. The signs $\pm u^{3}$ are known as focusing ($+$) and defocusing ($-$), respectively.
This reflects itself in the conserved energy
\EQ{
E(u,\dot u) = \int_{\R^{3}} \Bigl[ \frac12(|\nabla u|^{2}+|u|^{2}+|\dot u|^{2}) \mp  \frac14 u^{4}\Bigr]\, dx.
}
In either of these cases, short-time existence of smooth solutions is known, as well as 
global existence and scattering to zero for small data. For the defocusing case, one has global existence for all data as well as scattering in the energy class 
\EQ{
 \vec u(t):=(u(t),\dot u(t)) \in \HH:=H^{1}\times L^{2},} 
see Brenner~\cite{Bren1,Bren2}, Ginibre, Velo~\cite{GV1,GV2}, Morawetz, Strauss \cite{MoSt}, and Pecher~\cite{Pech}. As usual, 
 scattering of a solution $u$ to a static state $\fy$ refers to the following asymptotic behavior: there exists a solution $v$ of the free Klein-Gordon equation in the energy class such that 
\EQ{\label{scat def}
 \|\vec u(t)-\vec\fy-\vec v(t)\|_\HH \to 0,}
as either $t\to\I$ or $t\to-\I$ (depending on the context). 
When $\fy=0$, we sometimes say ``scattering'' instead of ``scattering to $0$''. 
See also Strauss~\cite{Strauss} and~\cite{ScatBlow} for a review of Strichartz estimates and wellposedness, as well as scattering in this setting.  

In this paper, we are only concerned with the focusing equation, for which smooth solutions can break down in finite time (which means that the $L^{\I}$-norm becomes infinite in finite time). More generally, the blow-up is defined for weak solutions in the energy class by divergence of the energy norm $\HH$ (which is more natural in view of the wellposedness theory, but these two definitions are equivalent for classical bounded\footnote{One can easily construct smooth global solutions with finite energy norm but with infinite $L^\I$ norm from the beginning, by superposing very sparse sequence of concentrated waves.} solutions). 
In fact, this happens for any data of negative energy, as was shown by Levine~\cite{Levine}.  
A natural question is then to decide which data lead to finite time blow-up versus scattering to zero. A first step in this direction was undertaken by Payne and Sattinger~\cite{PS}. 
Their result is formulated in terms of the so-called ground state, which is the unique positive radial decaying solution of 
\[
-\Delta Q + Q = Q^{3},
\]
see \cite{Strauss77,BerLions,Coff}. Amongst all stationary solutions of~\eqref{eq:NLKG3} in $H^{1}$, $\pm Q$ are unique  (up to translations)  with the
property that they minimize the action (or stationary energy)
\[
J(\fy) = \int_{\R^{3}} \Bigl[\frac12(|\nabla\fy|^{2}+|\fy|^{2})-\frac14|\fy|^{4} \Bigr] \, dx
\]
A scaling functional associated with $J$ is defined by 
\[
K(\fy)= \int_{\R^{3}} \big[ |\nabla\fy|^{2}+|\fy|^{2}-|\fy|^{4} \big]\, dx
\]
The significance of $K$ lies with the fact that 
\[
\inf\{ J(\fy)\mid K(\fy)=0,\ \fy\in H^{1}\setminus\{0\} \} = J(Q)
\]
As observed by Payne and Sattinger, this easily implies that the regions
\EQ{\label{eq:PSintro}
\PS_{+}&=\{ (u_{0}, u_{1})\in\HH\mid E(u_{0}, u_{1})<J(Q),\; K(u_{0})\ge0\} \\
\PS_{-}&=\{ (u_{0},  u_{1})\in\HH\mid E(u_{0}, u_{1})<J(Q),\; K(u_{0})< 0\} 
}
are invariant under the nonlinear flow in the phase space $\HH$. 
Moreover, they showed that solutions 
in~$\PS_{+}$ are global (in both time directions), whereas those in~$\PS_{-}$ blow up in finite time (in both time directions).  
In particular, the stationary solution $Q$ is unstable, see
also Shatah~\cite{Sha85} and Berestycki, Cazenave~\cite{BerCaz}. Scattering in $\PS_{+}$ was only recently shown by Ibrahim, Masmoudi, and Nakanishi~\cite{ScatBlow},
using a Kenig-Merle type argument originating from \cite{KM1,KM2}. 

This paper addresses the behavior of solutions whose energy exceeds that of $J(Q)$ by a small amount. More precisely, let $\vec u:= (u,\dot u)$ and set 
\EQ{ \label{def He}
 \HH^\e:=\{\vec u\in \HH \mid E_{m}(\vec u)<E(Q,0)+\e^2\}}
 with the {\em minimal energy} 
 \EQ{\label{eq:Em}
 E_{m}(\vec u):= |E(\vec u)^{2} - P(\vec u)^{2}|^{\frac12} \sign( E(\vec u)^{2} - P(\vec u)^{2} ) , 
 }
 where $P(\vec u) = \lan \dot u|\nabla u\ran$ is the conserved momentum. In contrast to $E$, the minimal energy $E_{m}$ is Lorentz invariant, 
 which makes it the right quantity to use in the nonradial case.  We call a solution with zero momentum {\em normalized}. 
 To every solution with $|E(\vec u)|>|P(\vec u)|$  there exists exactly one Lorentz transform\footnote{For subtle issues such as  the meaning of the Cauchy problem
 for Lorentz transformed solutions, and general wellposedness questions of transformed solutions,  see Section~\ref{sec:dyn} below.} 
 which reduces it to a normalized one, see~\eqref{eq:EPtrans}. If $|E(\vec u)|< |P(\vec u)|$, then there exists a Lorentz transform $L$ such that 
 $E(\vec u\circ L)<0$. But such solutions are known to blowup in finite positive and negative times, see Payne-Sattinger~\cite{PS}. If $|E(\vec u)|=|P(\vec u)|$, then along some sequence
 of Lorentz transforms $L_{j}$ we have $E(\vec u \circ L_{j})\to0$ as $j\to\I$. But then either $K(u\circ L_{j})<0$ for some $j$, which means that $u$
 blows up in both time directions by~\cite{PS}, or $u\circ L_{j}\to0$ strongly
 in $H^{1}$ so that $u$ exists globally. Thus, we can always talk about normalized solutions in those cases where~\cite{PS} does not apply, which are
 the ones relevant to this paper. 

\begin{thm}\label{thm:main}
Consider all solutions of NLKG~\eqref{eq:NLKG3} with initial data $\vec u(0)\in\HH^\e$ for some small $\e>0$. We prove that the solution set is decomposed into nine non-empty sets characterized as 
\begin{enumerate}
\item Scattering to $0$ for both  $t\to\pm\I$, 
\item Finite time  blow-up on both sides $\pm t>0$, 
\item Scattering to $0$ as $t\to\I$ and finite time  blow-up in $t<0$, 
\item Finite time  blow-up in $t>0$ and scattering to $0$ as $t\to-\I$, 
\item Trapped by $\pm Q$ for $t\to\I$ and scattering to $0$ as $t\to-\I$, 
\item Scattering to $0$ as $t\to\I$ and trapped by $\pm Q$ as $t\to-\I$, 
\item Trapped by $\pm Q$ for $t\to\I$ and finite time  blow-up in $t<0$, 
\item Finite time  blow-up in $t>0$ and trapped by $\pm Q$ as $t\to-\I$, 
\item Trapped by $\pm Q$  as $t\to\pm\I$, 
\end{enumerate}
where ``trapped by $\pm Q$" means that the normalized  solution stays in a $O(\e)$ neighborhood of $$\{\pm Q(\cdot+y)\mid y\in\R^{3}\}$$ forever after some time (or before some time). 
The initial data sets for (1)-(4), respectively, are open. 
\end{thm}

The radial case of this theorem was proved in~\cite{NLKGrad}. We follow the same approach as in~\cite{NLKGrad}, but with the added feature
of having to control a small translation vector, after the Lorentz transform to the normalized solution. 
As in the radial case, Theorem~\ref{thm:main} extends to all dimensions $x\in\R^d$ and equations of the form 
\[
\Box u + u = |u|^{p-1}u, \quad 1+\frac{4}{d}<p<1+\frac{4}{d-2},
\]
namely those powers of $L^2$-supercritical and $H^1$-subcritical. 
 The main difference between~\cite{NLKGrad} and this one lies with the further analysis of the trapped solutions in
Theorem~\ref{thm:main}. Recall that we proved in the radial case (where all solutions are normalized) that  any solution in $\HH^{\e}$ which is trapped by $Q$ as $t\to\I$ 
scatters to $Q$ as $t\to\I$,  see~\eqref{scat def}.  

This is the ``conditional asymptotic stability behavior'' observed in \cite{S,KrS1,Bec2} for the nonlinear Schr\"odinger equation (NLS), and near the equilibrium $Q$  these states
form a smooth codimension one manifold which is of the center-stable type, see~\cite{BJ}. 
Note that there are no modulation parameters in the radial case, such as translation or scaling, which simplifies the analysis considerably. 
In contrast, in the nonradial case,  $Q$ may be translated as well as  Lorentz transformed,
leading to a $6$-dimensional manifold of solitons (which consists of two disconnected components corresponding to $\pm Q$), parametrized by the relativistic momentum $p\in\R^3$ and the center of mass $q\in\R^3$. The traveling waves also undergo the {\em Lorentz contraction} (i.e., they are flattened in one direction), and are  explicitly given at each $(p,q)\in\R^6$ in the form 
\EQ{
\label{eq:Qpq}
 Q(p,q)=Q(x-q+p(\LR{p}-1)|p|^{-2}p\cdot(x-q)), \quad p,q\in \R^{3},}
where $\LR{p}:=\sqrt{1+|p|^2}$. The ground-state traveling waves are those solutions 
\EQ{
 u(t)=\pm Q(p,q(t)), \pq p\in\R^3,\pq \dot q(t)=\frac{p}{\LR{p}},} 
with fixed momentum $p$ and velocity $p/\LR{p}$. In our context, $p$ will be small for the normalized solutions, whereas $q$ can be arbitrary. 

The above scenario is of course the analogue of the center-stable manifolds of~\cite{S,Bec2} for NLS, where the family of solitons has 8 dimensions, parametrized by the non-relativistic momentum in $\R^3$, the center of mass in $\R^3$, the gauge in $S^1$ and the charge in $\R_{+}$. However, the NLS case of the
latter two references is quite different from the wave case treated here, since the symmetry group of NLS (generated by Galilei transforms, translation, 
modulation, and scaling) does not mix space and time in the way that  Lorentz transforms do. In addition, the modulation equations 
 in the NLS case are first order ODEs, whereas in the wave case they at least appear on first sight to be of the second order. A considerable amount of effort is therefore expended here on
setting up a suitable framework for the modulational and scattering analysis, see Sections~\ref{sec:complex} and~\ref{sec:CS}. 
In particular, we will set things up in such  a way that the modulation equations are of the first order. 

We can summarize the findings of this paper concerning the trapped solutions in Theorem~\ref{thm:main} as follows: 

\begin{thm}\label{thm:main2}
All normalized solutions in $\HH^{\e}$ which are trapped by $Q$ as $t\to\I$ lie on a smooth, codimension-one manifold $\calM_{+}(Q)$ in~$\HH$,  
and those trapped by $-Q$ as $t\to\I$ lie on $\calM_{+}(-Q)=-\calM_{+}(Q)$.  Analogously, those trapped by $\pm Q$ as $t\to-\I$ lie on $\calM_{-}(\pm Q)$. 
The manifold $\calM_{+}(Q)$ is invariant under the flow of~\eqref{eq:NLKG3} in forward time, and all solutions starting on $\calM_{+}(Q)$ scatter to 
the six-dimensional manifold generated by $Q$ as $t\to\I$ in the following sense:  there exist $p_\I\in\R^3$ and $C^{1}$ paths $q(t)\in\R^3$ with the property that 
\EQ{ \label{u Q scat}
 \dot q(t)\to p_\I/\LR{p_\I}, \pq u(t) - Q(p_\I,q(t)) \text{ scatters to } 0, \pq(t\to\I).} 

The center-stable manifolds $\calM_{+}(Q)$ and $\calM_{-}(Q)$ intersect transversely in a codimension two manifold, the center manifold, which characterizes case~(9)
of Theorem~\ref{thm:main}.  

Finally, all normalized solutions $\vec u(t)$ with energy $E(\vec u)=E(Q,0)$, and which are trapped by $\pm Q$ in positive times, are characterized as follows: either $u$ is constant
and equal to some translate of $\pm Q$, or for some $t_{0}\in\R$, $x_{0}\in \R^{3}$, 
\EQ{\label{eq:DuyM}
u(t,x)= \pm W_{+}(t+t_{0}, x+x_{0}) \text{\ \ or\ \ } u(t,x)= \pm W_{-}(t+t_{0}, x+x_{0}) 
}
where $W_{\pm}$ are solutions of~\eqref{eq:NLKG3} which approach $Q$ exponentially fast in $\HH$, and either blowup or scatter to $0$ as $t\to -\I$. An analogous statement holds for the negative time direction. 
\end{thm}

For more detailed asymptotic behavior of in Theorem~\ref{thm:main2}, see Section~\ref{sec:CS} below. 
The solutions $W_{\pm}$ are the analogue of the threshold solutions found by Duyckaerts, Merle~\cite{DM1,DM2} for the nonlinear wave and Schr\"odinger equations with the $H^1$ critical power.  As in the radial case~\cite{NLKGrad}, our proof of Theorem~\ref{thm:main2}
relies on the gap property of the linearized operator 
\EQ{ \label{def L+}
 L_{+}:=-\Delta+1-3Q^{2},}
i.e., $\spec(L_+)\cap(0,1]=\emptyset$ without threshold resonance at $1$.  
This property was verified in~\cite{DS} by a numerically assisted  Birman-Schwinger method, whereas an analytical approach to the study of such spectral problems is developed in~\cite{CHS}.  

The definition of center-stable manifolds given by Bates and Jones~\cite{BJ} does not refer to the asymptotic behavior of the solutions originating
on them. Rather, these manifolds are defined by their local in time invariance, their tangent space at the equilibrium, and the fact that they are transverse to the unstable manifold
(and they do not need to be unique with these properties), see also Hirsch, Pugh, Shub~\cite{HPS}. 
In particular, the equilibrium must be a true stationary solution which is not the case in this paper due to the symmetries of~\eqref{eq:NLKG3}. This is why~\cite{BJ}
only applies to the NLKG equation~\eqref{eq:NLKG3} in the radial setting. Moreover, since the method of proof in~\cite{BJ} relies entirely on energy estimates and does
not use any dispersive properties of the nonlinear flow, no asymptotic stability properties of the solutions starting on the center-stable manifold are obtained. However, 
using the action $J$ as a Lyapunov functional, Bates and Jones do show that their center manifold is orbitally stable in both time directions (in particular, solutions
starting on the center manifold exist for all  times). 

In contrast, the point of view of \cite{S,KrS1,Bec2} is that of asymptotic stability theory of solitons. Since the equations under considerations in these papers, 
such as the cubic NLS equation in three dimensions, are orbitally unstable and in fact small perturbations of solitons can lead to finite time blowup, any asymptotic
stability properties one hopes to recover need to be {\em conditional}.  Since the ground state soliton in these settings is known to have exactly one exponentially unstable mode,
the condition turns out to be a codimension one manifold. One advantage of this approach is that it is not restricted to the radial setting and allows for moving parameters.
The latter always correspond to symmetries of the equation and are controlled by suitable orthogonality conditions relating to the discrete spectrum of the linearized operators. 

Working in an invariant topology such as the energy class, Beceanu~\cite{Bec2}  was able to show that the  center-stable manifold obtained by the asymptotic
stability approach\footnote{Which is called Lyapunov-Perron method, see~\cite{V} and~\cite{Ball}, in contrast to the Hadamard method used by Bates, Jones~\cite{BJ}.} 
has the properties required by~\cite{BJ}. Strictly speaking, due to the underlying symmetries present in the Schr\"odinger equation (even in the radial setting one has the modulation
and dilation symmetries, see~\cite{NakS2}) or the NLKG equation~\eqref{eq:NLKG3}, a direct comparison between the manifolds of~\cite{BJ} and those in~\cite{Bec2}, \cite{NakS2}, or
of Theorem~\ref{thm:main2} above,  is inadmissible. However, passing to the quotients by the symmetries does allow for a comparison in some sense. 
For example, $W_{+}, W_{-}$ and $Q$ form the stable manifold associated with $Q$.

\section{The Complex Formalism}\label{sec:complex}

We consider the nonlinear Klein-Gordon equation 
\EQ{\label{eq:NLKG}
 \ddot \uu - \De \uu + \uu = \uu^3, \pq \uu(t,x):\R^{1+3}\to\R,}
in the energy space $(\uu(t),\dot \uu(t))\in \HH=H^1(\R^3)\times L^2(\R^3)$. 
It is convenient to formulate the equation in terms of the complexified variable 
\EQ{ \label{complexify}
 u := \D \uu - i\dot \uu, \pq \D:=\sqrt{1-\De}.}
Using the notation
\EQ{ \label{vec component}
 u_1:=\D^{-1}\Re u, \pq u_2:=\Im u}
  for any complex function $u$, 
the conserved energy becomes
\EQ{
 E(u) = \int_{\R^3} \Bigl[\frac{|u_2|^2+|\na u_1|^2+|u_1|^2}{2}-\frac{|u_1|^4}{4}\Bigr] \, dx = \|u\|_{L^2}^2/2 - \|u_1\|_{L^4}^4/4,}
and the energy (or phase) space for $u(t)$ is the $\R$-Hilbert space $\HL$ defined by 
\EQ{
 \HL := L^2(\R^3;\C), \pq \LR{f|g}:=\Re\int_{\R^3}f(x)\bar{g(x)} \, dx\ (f,g\in\HL).}
We can rewrite \eqref{eq:NLKG} as the system consisting  of \eqref{vec component} and 
\EQ{ \label{eq in vec u}
 u_t = i\D u - i u_1^3.}
In the following, the relation \eqref{vec component} is always assumed for any vector expression, 
but not necessarily \eqref{complexify}. 
The ground state in this formulation is denoted by 
\EQ{
 \VQ:=\D Q.}
We decompose any solution $u$ of \eqref{eq in vec u} in the form
\EQ{
 u = (\VQ+ v)(x-c(t)).}
Then the equation for $v$ is 
\EQ{\label{eq:vPDE} 
 v_t = i\D\LL v + \p_k(\VQ+ v) \dot c^k - i N(v_1),}
where $\LL$ and $N$ are defined by 
\EQ{\label{eq:LNdef}
 \pt \LL v = v - 3\D^{-1}Q^2v_1, 
 \pq N(v_1) = 3Qv_1^2 + (v_1)^3 = O(\|v\|_2^2),}
 respectively. 
 Note that $v$ does not satisfy \eqref{complexify} in general, due to the time-dependent translates $c(t)$.
 However, this will not affect the argument.

$\LL$ is a self-adjoint $\R$-linear Fredholm operator on $\HL$. 
The linearized operator $i\D\LL$ has the following generalized eigenfunctions 
\EQ{\label{eq:geneifunc}
 i\D\LL\na\VQ=0,\pq i\D\LL i\na Q=-\na\VQ, \pq i\D\LL g_\pm=\pm kg_\pm,}
where $g_\pm$ is given in terms of the ground state $\ro$ of $L_+=\D^2-3Q^2$: 
\EQ{
 g_\pm = (2k)^{-1/2}\D\ro \mp i (k/2)^{1/2}\ro, \pq (L_+\ro=-k^2\ro, \ \|\ro\|_2=1,\ \ro>0).}
The natural symplectic form associated with this complex formalism is 
\EQ{\label{eq:sympl}
 \om(u,\tilde u) := \LR{i\D^{-1}u| \tilde u}.}
 Indeed, in keeping with the classical Hamiltonian formalism we need to be able to write
 the equation~\eqref{eq in vec u} in the form
 $
 \dot u = X_{E}(u)
 $
 where $X_{E}(u)$ is the Hamiltonian vector field associated with $E$. This means that 
 \[
 dE(u)(\cdot)= \omega(\cdot, X_{E}(u))
 \]
 or in other words,
 \[
 \lan u-\D^{-1}u_{1}^{3}|\cdot \ran = \omega(\cdot ,i\D u- iu_{1}^{3})
 \]
 which yields~\eqref{eq:sympl}. 
We have, for $\al,\be=1,2,3$, 
\EQ{
 \pt\om(\p_\al\VQ,i\p_\be Q)=\de_{\al,\be}\|\p_1Q\|_2^2=\de_{\al,\be}J(Q),
 \pq\om(g_\pm,g_\mp)=\pm 1, 
 \pr\om(\na\VQ,g_\pm)=\om(i\na Q,g_\pm)=0.}
We now perform the following symplectic decomposition of $v$:
\EQ{ \label{decop v}
 \pt v = \la_+g_+ + \la_-g_- + \ga, 
 \pq \la_\pm := \om(v,g_\mp)/\om(g_\pm,g_\mp)=\om(v,\pm g_\mp),}
 which implies that
 \[
 \om(\ga,g_{\pm})=0 \iff \lan \ga_{1}| \rho \ran = \lan \ga_{2} | \rho\ran=0
 \]
We can then expand the energy 
\EQ{
 2(E(u)-J(Q))\pt= \LR{L_+v_1|v_1}+\|v_2\|_2^2-2C(v)=\LR{\LL v| v}-2C( v)
 \pr=-2k\la_+\la_-+\LR{\LL\ga|\ga}-2C( v)
 \pr=2k(\la_2^2-\la_1^2)+\LR{\LL\ga|\ga}-2C( v),}
where $\la_j$ and $C$ are defined by 
\EQ{
 \pt \la_1=\frac{\la_++\la_-}{2}, \pq \la_2=\frac{\la_+-\la_-}{2}, 
 \pq C( v)=\LR{Q|v_1^3}+\frac{\|v_1\|_4^4}{4}=O(\|v_1\|_{H^1}^3).}

\section{Parameter choice}

We can reduce the number of coordinates by using the Lorentz invariance.  To be more specific,
let $w$ be a global strong energy solution of~\eqref{eq:NLKG} and consider 
the following example of a Lorentz transform of $w$ given by 
\EQ{ \label{Lorentz}
 w(t,x) \mapsto w_\nu(t,x):=w(t\cosh\nu+x_1\sinh\nu,x_1\cosh\nu+t\sinh\nu,x_2,x_3)}
for $\nu\in\R$. Then one checks that  $w_\nu$ is again a strong energy solution of~\eqref{eq:NLKG} and 
\EQ{\label{eq:EPtrans}
 \pt E( w_\nu) = E( w)\cosh\nu+P_1( w)\sinh\nu, 
 \pr P_1( w_\nu)=P_1( w)\cosh\nu+E( w)\sinh\nu, \pq P_\al( w_\nu)=P_\al( w),\ (\al=2,3),}
where $P=(P_1,P_2,P_3)$ denotes the total momentum 
\EQ{
 P( w) = \LR{w_t|\na w} = \frac12\om(u,\nabla u),}
 cf.~\eqref{complexify}. 
Note that the Lorentz transform preserves $E^2-|P|^2$. 
Every solution with finite energy $|E|>|P|$ is transformed, by a unique element of the Lorentz group, to another solution with zero momentum:
\EQ{ \label{momentum cond}
 0=P( u) = \om( u,\na u)/2=\om( v,\na\VQ)+\om( v,\na v)/2,}
which minimizes the energy among the family of solutions generated by the Lorentz transformation group. 
By the discussion in Section~\ref{sec:intro} the cases $|E|\le|P|$ are covered by~\cite{PS} and can be ignored. 
Since the dynamical properties such as scattering and blowup are not changed by the Lorentz transforms, 
we may restrict our dynamical analysis to the invariant subset
\EQ{
 \HL_0 := \{ u\in \HL \mid \om(u,\na u)=0\}.}
 Note that this is well-defined since $\om$ gains a derivative. 
However, we will keep the generality of $\HL$ in the static analysis, see Lemma~\ref{K lower bd} below. 
We define the decomposition 
\EQ{
  u = \sg(\VQ+ v)(x-c), \pq \sg=\pm 1,\ c\in\R^3}
by the orthogonal projection in $H^{-1}(\R^3)$, or the minimization 
\EQ{ \label{L2 mini}
 \|u_1-\sg Q(x-c)\|_{L^2}=\min_{b\in\R^3} \|u_1\mp Q(x-b)\|_{L^2},}
which is attained for any $ u\in\HL$, and uniquely so if the right-hand side is small enough. 

The minimization implies the orthogonality conditions 
\EQ{ \label{orth}
 0=\om( v,i\p_j Q) = \LR{v_1|\p_jQ} \pq (j=1,\dots,3).} 
Differentiating \eqref{orth} with respect to $t$, we obtain the parameter evolution in $\HL_0$ 
\EQ{ \label{eq c}
 0\pt=\p_t\om( v,i\p_\al Q) \pq(\al=1,2,3)
 \pr=\om(i\D\LL  v + \dot c\cdot (\VQ+ v) - i N(v_1),i\p_\al Q)
 \pr=\|\p_1Q\|_{L^2}^2\dot c-\LR{\na\p_\al Q|v_1}\cdot\dot c  - \om( v,\na v)/2,}
where we used \eqref{momentum cond} to rewrite the first term of the second line, while the nonlinear 
term $N$ does not contribute because of its reality. Thus,  we obtain 
\EQ{ \label{bd on ct}
 |\dot c| \lec |\om( v,\na v)| \lec \| v(t)\|_2^2,}
as long as $\|v_1\|_{L^2}\ll 1$ and the solution $ u\in\HL_0$. \eqref{orth} and the symplectic 
decomposition \eqref{decop v} imply the orthogonality relations 
\EQ{
 \pt 0=\LR{\ga_1|\p_jQ}=\LR{\ga_1|\ro}=\LR{\ga_2|\ro}.}
Since $\{\ro,\na Q\}$ covers the non-positive eigenfunctions of $L_+$ (see \cite[Appendices A and E]{Wein}), 
\EQ{
 \LR{\LL\ga|\ga}=\LR{L_+\ga_1|\ga_1}+\|\ga_2\|_2^2\simeq\|\ga_1\|_{H^1}^2+\|\ga_2\|_2^2\simeq\|\ga\|_2^2.}
The linearized energy norm $\| v\|_E\ge 0$ is defined in the subspace \eqref{orth} by 
\EQ{
 \| v\|_E^2 = k(\la_1^2+\la_2^2)+\LR{\LL\ga|\ga}/2 = \frac{k}{2}(\la_+^2+\la_-^2)+\LR{\LL\ga|\ga}/2 \simeq \| v\|_2^2.}
In fact, we have 
\EQ{
 \| v\|_E \simeq \inf_{\pm, b}\| u\mp \VQ(x-b)\|_{L^2},}
since 
\EQ{\label{eq:minbc}
 \|v_1\|_{H^1}- \|u_1\mp Q(x-b)\|_{H^1} 
 \pt\le \|\pm Q(x-b)-Q(x-c)\|_{H^1}
 \pr\simeq\|\pm Q(x-b)-Q(x-c)\|_{L^2} 
 \pr\le \|u_1\mp Q(x-b)\|_{L^2}+\|u_1-Q(x-c)\|_{L^2}
 \pr\le 2\|u_1\mp Q(x-b)\|_{L^2}}
and so, $\|v_1\|_{H^1}\lec \|u_1\mp Q(x-b)\|_{H^1}$. 
Moreover, 
\EQ{
 \| v\|_E^2 = E( u)-J(Q)+2k\la_1^2+C( v),\pq C( v)=o(\| v\|_{H^1}^2).}
We thus obtain the following lemma, which introduces the nonlinear distance function
from~\cite{NLKGrad} in the non-radial setting. 
 
\begin{lem}
There exists $\de_E>0$ and $d_Q(u):\HL\to[0,\I)$ continuous such that 
\EQ{ \label{energy dist}
 \pt d_Q( u)\simeq \inf_{\pm, b} \| u\mp\VQ(b)\|_{L^2}, 
 \pr d_Q( u) \le \de_E \implies d_Q( u)=E( u)-J(Q)+2k\la_1^2.}
Moreover, \eqref{L2 mini} has a unique solution $(\sg,c)$ for $d_Q( u)\le\de_E$, and decomposing 
\EQ{
 \pt  u=\sg(\VQ+ v)(x-c), \pq  v = \la_+g_+ + \la_-g_-+\ga,
 \pq \la_\pm=\om( v,\pm g_\mp),}
we have 
\EQ{
 d_Q( u)^2 \simeq \| v\|_E^2 = \frac{k}{2}(\la_+^2+\la_-^2)+\frac 12 \LR{\LL\ga|\ga}.} 
In addition, if 
\EQ{ \label{outer region}
 2(E( u)-J(Q))<d_Q( u)^2<\de_E^2}
then  $d_Q( u)\simeq|\la_1|=|\la_++\la_-|/2$.
\end{lem}

\section{Hyperbolic and variational estimates}

Next we investigate the hyperbolic structure in $\HL_0$. The evolution of $\la$ is obtained
by differentiating~\eqref{decop v}. Via~\eqref{eq:vPDE} this yields 
\EQ{
 \dot\la_\pm\pt=\om(i\D\LL  v + \dot c\cdot \na(\VQ+ v) - i N(v_1),\pm g_\mp)
 \pr=\pm [k\la_\pm + (2k)^{-1/2}\LR{N(v_1)|\ro}-\om( v,\na g_\mp)\cdot\dot c],}and so
\EQ{
 \pt\dot\la_1=k\la_2-\sqrt{\frac{k}{2}}\LR{v_1|\na\ro}\cdot\dot c,
 \pr\dot\la_2=k\la_1+\frac{1}{\sqrt{2k}}\LR{v_2|\na\ro}\cdot\dot c+\frac{1}{\sqrt{2k}}\LR{N(v_1)|\ro}.}
 Both of these equations exhibit the hyperbolic nature of the ODE for $(\la_{+},\la_{-})$ or $(\la_{1},\la_{2})$. 
Hence
\EQ{
 \p_t d_Q( u)=4k\la_1\dot\la_1=4k^2\la_1\la_2+O(\la_1\| v\|_E^3)}
  in the region $d_Q( u)<\de_E$. 
Moreover, in \eqref{outer region}, one has  $\la_1 \simeq -\sg d_Q( u)$, with $\sg=\pm 1$. If in addition $\p_t d_Q( u)\ge 0$, then 
\EQ{
 \sg\la_1 \simeq \sg\la_+ \ge \sg\la_- - O(\la_1^3).}
Let $ v_g:=\la_+g_++\la_-g_-$. Then the projected energy 
\EQ{
 E_g( u) := J(Q)-k\la_+\la_- - C( v_g)}
solves the equation 
\EQ{
 \p_t E_g( u) \pt= -k\dot\la_+\la_--k\la_+\dot\la_- - \LR{N(v_{g1})|\p_t v_{g1}}
 \pr= \LR{N(v_1)-N(v_{g1})|\p_t v_{g1}} -\om( v,\na\p_t v_g)\cdot\dot c
 \pn= O(\|\ga\|_2\| v\|_2^2+\| v\|_2^4),}
in the region \eqref{outer region}, and also 
\EQ{
 E( u)-E_g( u)=\LR{\LL\ga|\ga}/2-C( v)+C( v_g)\simeq\|\ga\|_2^2+O(\|\ga\|_2\| v\|_2^2).}
Therefore, in the region \eqref{outer region} we have  
\EQ{ \label{ga bd}
 \|\ga\|_{L^\I_tL^2_x(0,T)}^2 \lec \|\ga(0)\|_2^{2} + \|\ga\|_{L^\I_tL^2_x(0,T)}\|\la\|_{(L^\I\cap L^2)_t(0,T)}^2 + \|\la\|_{L^4_t(0,T)}^4.}

Next we compute the leading term in the kinetic functionals, defined by 
\EQ{
 \pt K_0(\fy) := \int_{\R^3} \Bigl[ |\na \fy|^2+|\fy|^2-|\fy|^4 \Bigr] dx, 
 \pq K_2(\fy) := \int_{\R^3} \Bigl[ |\na \fy|^2-\frac{3}{4}|\fy|^4 \Bigr] dx.}
Inserting $u=\VQ+v$ and using the equation of $Q$, we obtain 
\EQ{
 \pt K_0(u_1)=-2 \LR{ Q^{3}|v_1}+O(\|v_1\|_{H^1}^2), 
 \pr K_2(u_1) = -\LR{Q^{3}-2Q|v_1}+O(\|v_1\|_{H^1}^2).}
Since $v_1=\sqrt{\frac{2}{k}}\, \la_1\ro+\ga_1$,  and $L_{+}Q=-2Q^{3}$, $L_{+}\ro=-k^{2}\ro$,  we have 
\EQ{ \label{exp K}
 \pt K_0(u_1)=  - k^{2} \sqrt{\frac{2}{k}} \LR{Q |\ro}\la_1-\LR{2Q^3|\ga_1}+O(\|v_1\|_{H^1}^2)
 \pr K_2(u_1)= -  (2+k^{2}/2)\sqrt{\frac{2}{k}} \LR{Q|\ro}\la_1-\LR{2Q+Q^3|\ga_1}+O(\|v_1\|_{H^1}^2).}
We can now formulate and prove the following nonradial version of the {\em ejection lemma} from~\cite{NLKGrad}. 

\begin{lem} \label{lem:eject}
There exist constants $\de_X\in(0,\de_E)$, $C_*\ge 1$ with the following properties: 
Let $u$ be a local solution of the NLKG equation~\eqref{eq in vec u}  on $[0,T]$ which belongs to $\HL_{0}$ and 
such that 
\EQ{
 R:=d_Q( u(0))\le \de_X, \pq E( u)<J(Q)+R^2/2,}
and so that for some $t_0\in(0,T)$, 
\EQ{ \label{exiting cond}
 d_Q( u(t))\ge R \pq (0<\forall t<t_0).}
Then $u$ extends as long as $d_Q( u(t))\le\de_X$, and meanwhile,
\EQ{
 \pt d_Q( u(t)) \simeq -\sg\la_1(t) \simeq -\sg\la_+(t) \simeq e^{kt}R,
 \pr |\la_-(t)|+\|\ga(t)\|_{L^2} \lec R+(e^{kt}R)^2,
 \pr \sg K_s(u_1(t)) \gec d_Q( u(t))-C_*d_Q( u(0)),}
for $s=0,2$ and with either $\sg=1$ or $\sg=-1$. 
Moreover, $d_Q( u(t))$ is increasing for $t\gg R^2$, and $d_Q( u(t'))\ge d_Q( u(t))-O(R^5)$ for $0\le t<t'\lec R$. 
\end{lem}

\begin{proof}
We have shown $d_Q( u)\simeq-\sg\la_1$, as long as $R\le d_Q( u)\le\de_E$. 
\eqref{exiting cond} implies $\p_t d_Q^{2}( u)|_{t=0}\ge 0$, and so $\la_+(0)\simeq\la_1(0)$. Integrating the equation for $\la_\pm$ yields
\EQ{
 |\la_\pm-e^{\pm kt}\la_\pm(0)|\pt \lec \int_0^t e^{k(t-s)}[|\LR{N(v_1)|\ro}|+\| v\|_{L^2}|\dot c|](s)\, ds
 \pr\lec \int_0^te^{k(t-s)}|\la_1(s)|^2\, ds,}
from which by continuity in $t$ we deduce that as long as $Re^{kt}\ll 1$, 
\EQ{
 \la_1(t)\simeq\la_+(t) \simeq -\sg Re^{kt}, \pq |\la_\pm(t)-e^{\pm kt}\la_\pm(0)|\lec (Re^{kt})^2.}
The estimate on $\ga$ follows from this and \eqref{ga bd}. 

The equation for $\la_2$ together with $\p_td_Q( u)(0)\ge 0$ implies that $-\sg\la_2\gec R(e^{kt}-1)-O(R^3)$.
Hence, for $t\gg R^2$ we have $-\sg\la_2\gec R$ and $\p_td_Q( u)>0$. 
For $0<t<t'\lec R$, we have $\p_td_Q( u)\gec-R^4$ and so $d_Q( u(t'))\ge d_Q( u(t))-O(R^5)$. 
The estimate on $K_s$ follows from the dominance of $\la_1$, together with \eqref{exp K}. 
\end{proof}

Next, we formulate the important variational lower bound from~\cite{NLKGrad}. Note that in this
case we work with $\HL$ and not~$\HL_{0}$. 

\begin{lem}[Variational lower bound] \label{K lower bd}
For any $\de>0$, there exist $\e_0(\de), \ka_0, \ka_1(\de)>0$ such that for any $u\in \HL$ satisfying 
\EQ{ \label{energy region}
 E( u)< J(Q)+\e_0(\de)^2, \pq \inf_{\pm,b}\| u\mp  \VQ(\cdot+b)\|_{L^2} \ge \de,}
one has either
\EQ{ \label{-K bd}
 K_0(u_1) \le -\ka_1(\de) \pq and \pq K_2(u_1) \le -\ka_1(\de),}
or 
\EQ{ \label{+K bd}
 K_0(u_1) \ge \min(\ka_1(\de), \ka_0\|u_1\|_{H^1}^2) \pq and \pq K_2(u_1) \ge \min(\ka_1(\de),\ka_0\|\na u_1\|_{L^2}^2).}
\end{lem}
\begin{proof}
The proof is essentially the same as in the radial case \cite{NLKGrad}. The only major difference is that we cannot use the compact
 imbedding  $H^{1}_{\mathrm{rad}}\hookrightarrow L^4_x(\R^{3})$; instead, we shall use  the concentration compactness method of Lions, \cite{PL1}, \cite{PL2}.
 
We  first prove the statement separately for $K_0$ and $K_2$, by contradiction. Let $ u^n\in\HL$ be a sequence 
satisfying~\eqref{energy region} with $\e_0=1/n$ but neither \eqref{-K bd} nor \eqref{+K bd}, with either $s=0$ or $s=2$ fixed. 
In particular, $K_{s}(u^{n}_1)\to0$ as $n\to\I$. Since $E(u^{n})$ is uniformly bounded, we conclude that $\{u^{n}_{1}\}_{n=1}^{\I}$ is uniformly
bounded in $H^{1}$. 
First, if $\| u_1^{n}\|_{2}\to0$, then  by Sobolev imbedding also $\|u^{n}_{1}\|_{4}\to0$ as $n\to\I$ whence also $\| \nabla u_1^{n}\|_{2}\to0$. But then 
the lower bound in~\eqref{+K bd} does hold, which is a contradiction. Hence, we may assume that $\|u_{1}^{n}\|_{2}\to c_{0}>0$ as $n\to\I$. 

We apply the 
concentration compactness argument to the bounded sequence $u^n_1\in H^1(\R^3)$.  In the {\em compactness case}, there exists a sequence $y_{n}\in \R^{3}$
so that
\[
u_{1}^{n}(\cdot + y_{n}) \to u_{\I} \text{\ \ strongly in \ \ } L^{2}\cap L^{4}
\]
as $n\to\I$. 
Let
\EQ{
 \pt G_0(\fy):=J(\fy)-K_0(\fy)/4=\|\fy\|_{H^1}^2/4, 
 \pr G_2(\fy):=J(\fy)-K_2(\fy)/3=\|\na\fy\|_{L^2}^2/6+\|\fy\|_{L^2}^2/2}
But then $\nabla u^{n}_{1} \rightharpoonup \nabla u_{\I}$ weakly in $L^{2}(\R^{3})$, and 
\[
G_{s}(u_{\I})\le J(Q), \quad K_{s}(u_{\I})\le 0
\]
and $u_{\I}\ne0$ (the latter holds since $u_{\I}=0$ would entail that $\|\nabla u^{n}_{1}\|_{2}\to0$, which leads to a contradiction as before). But then $u_{\I}=\pm Q(\cdot + b)$
for some $b\in\R^{3}$ by the variational characterization of $Q$ (see, e.g.~\cite[Lemma~2.4]{ScatBlow}).   
This means that $\|\nabla u_{1}^{n} \|_{2}\to \|\nabla u_\I\|_{2}=\|\nabla Q\|_{2}$, whence also $u_{1}^{n} \to u_{\I}$ strongly in $H^{1}$. 
This contradicts~\eqref{energy region}. 

In the {\em vanishing case},  one has $\| u^{n}_{1}\|_{4}\to0$. But then $\|u_{1}^{n}\|_{H^{1}}\to0$, which leads to a contradiction as before. 

It remains to treat  the {\em dichotomy case}.  Thus, 
let $v_n$ and $w_n$ be the separating sequences. Since their $L^2$ norms converge to non-zero values, the positive functionals $G_0$ and $G_2$ 
converge to some values in $(0,J(Q))$, for both sequences. Since $Q$ is the minimizer of $G_s$ in the region $K_s\le 0$, we deduce that 
$K_s$ is positive for $v_n$ and $w_n$ for large $n$. On the other hand, $$\limsup[K_s(v_n)+K_s(w_n)]\le 0,$$ by the choice of $ u_n$, and 
so $K_s\to 0$ for both $v_n$ and $w_n$. Since $J(v_n)= G_s(v_n)+o(1)<J(Q)$, and the same for $w_n$ for large $n$, we deduce, from the 
lower bound on $K_s$ below the ground state energy~\cite[Lemma 2.12]{ScatBlow}, that $v_n$ and $w_n$ tend to $0$ in $\dot H^1$, and 
so in $L^4$. Finally,  so does $u^n_1$ which places us back in the vanishing case. 

After obtaining the conclusion separately for $s=0$ and $s=2$, the remaining proof in~\cite{NLKGrad} by connectedness 
works as well in the nonradial case, where $\la$ corresponds to $\sqrt{2}\la_1$ in this paper.  
\end{proof}

Just as in the radial case, the above two lemmas allow us to define the sign functional. We again restrict to~$\HL_{0}$. 

\begin{lem}[Sign functional]\label{lem:sign}
Let $\de_S:=\de_X/(2C_*)>0$ where $\de_X>0$ and $C_*\ge 1$ are the constants from~Lemma \ref{lem:eject}. Let $0<\de\le\de_S$ and 
\EQ{
 \HL_{(\de)}:=\{ u\in\HL_0 \mid E( u)<J(Q)+\min(d_Q( u)^2/2,\e_0(\de)^2)\},}
where $\e_0(\de)$ is given in Lemma~\ref{K lower bd}. Then there exists a unique continuous function $\Sg:\HL_{(\de)}\to\{\pm 1\}$ satisfying 
\EQ{
 \CAS{  u\in\HL_{(\de)},\ d_Q( u)\le\de_E &\implies \Sg( u)=-\sign\la_1,\\
  u\in\HL_{(\de)},\ d_Q( u)\ge\de &\implies \Sg( u)=\sign K_0(u_1)=\sign K_2(u_1),}}
where we set $\sign 0=+1$. 
\end{lem}
\begin{proof}
The proof is the same as in the radial case \cite{NLKGrad}. 
\end{proof}

\section{One-pass theorem}

Once we have obtained the hyperbolic and the variational estimates, the one-pass theorem 
is proved almost in the same way as in the radial case \cite{NLKGrad}. The only remaining, but minor, differences are
\begin{enumerate}
\item The center $c$ may be different for an ``almost homoclinic" orbit, between the departing time and the returning time.
\item In the hyperbolic region, the nonlinear distance function $d_Q( u)$ is not strictly convex in time.
\end{enumerate} 
For the issue (2), we have only to ignore small fluctuations around the ``bottom" of the hyperbolic trajectory. After choosing the small numbers $\e,R,\de_*>0$ as in \cite{NLKGrad}, let $u(t)$ be a solution on the maximal interval $I\subset\R$ satisfying for some $\t_1<\t_2<\t_3\in I$, 
\EQ{
 E( u)<J(Q)+\e^2, \pq \max_{j=1,3}d_Q( u(\t_j))<R<R+R^2<d_Q( u(\t_2)).}
Then there exist $T_1\in(\t_1,\t_2)$ and $T_2\in(\t_2,\t_3)$ such that 
\EQ{
 \pt d_Q( u(T_1))=R=d_Q( u(T_2))<d_Q( u(t)) \pq (T_1<t<T_2),
 \pr \max_{T_1<t<T_2} d_Q( u(t))>R+R^2.}
Let $\M$ be the totality of minimal points $t_m\in[T_1,T_2]$ of $d_Q( u(t))$ with a minimum $<\de_*$. By the ejection Lemma \ref{lem:eject}, we can extract a finite sequence $t_1<t_2<\cdots<t_n\in\M$ for some $n\ge 2$ such that $t_1=T_1$, $t_n=T_2$, and for each $j=1,\dots,n-1$, 
\EQ{
 \pt \max_{t_j<t<t_{j+1}}d_Q( u(t))\ge \de_X,}
and for some $t_{j+1/3}<t_{j+2/3}\in(t_j,t_{j+1})$, 
\EQ{ 
 \CAS{
 d_Q( u(t))\simeq e^{k|t-t_j|}d_Q( u(t_j)) &(t_j<t<t_{j+1/3}),\\
 d_Q( u(t_{j+1/3}))=\de_X=d_Q( u(t_{j+2/3})),\\
 d_Q( u(t))\ge R_* &(t_{j+1/3}<t<t_{j+2/3}),\\ 
 d_Q( u(t))\simeq e^{k|t-t_{j+1}|}d_Q( u(t_{j+1})) &(t_{j+2/3}<t<t_{j+1}).}}
Note that the assumption $d_Q( u(\t_2))>R+R^2$ was used to have at least one time of ejection up to $\de_X$ between $T_1$ and $T_2$, otherwise $u$ could be just fluctuating around $d_Q\simeq R$ on the whole interval $[T_1,T_2]$.

The modification required by the other issue (1) is also straightforward. 
We define the cut-off function $w(t,x)$ for the virial identity by
\EQ{
 w=\chi((x-c(T_1))/(t-T_1+S))\chi((x-c(T_2))/(t-T_2+S))}
for some $S>0$ satisfying $|\log R|\ll S\ll 1/R$, where $\chi(x)\in C_0^\I(\R^3)$ is a fixed radial function satisfying $\chi(x)=1$ for $|x|\le 1$ and $\chi(x)=0$ for $|x|\ge 2$. 
We have the localized virial identity 
\EQ{
 V_w(t):=\LR{wu_t|(x\na+\na x)u}, \pq \dot V_w(t)=-K_2(u_1(t))+O(\Ex(t)),}
where $\Ex$ denotes the exterior energy defined by 
\EQ{
 \pt \Ex(t)=\frac12 \int_{X_1(t)}[|\dot u_2|^2+|\na u_1|^2+|u_1|^2]\, dx,
 \pr x\in X(t) \iff |x-c(T_1)|>(t-T_1+S) \text{ or } |x-c(T_2)|>(T_2-t+S).}
Then  as in \cite{NLKGrad}, we have 
\EQ{
 \Ex(t) \lec \Ex(T_1)+\Ex(T_2) \lec R^2 \pq (T_1<t<T_2),}
and so
\EQ{
 \dot V_w(t) = -K_2(u_1(t)) + O(R^2) \pq (T_1<t<T_2).}
In conclusion, by the same argument as in \cite{NLKGrad} we arrive at the following no-return statement. 

\begin{thm}[One-pass theorem]
There are constants $\e_*,R_*>0$ such that $2\e_*<R_*<\de_X$ with the following property: If $u\in C(I;\HL)$ is a solution of NLKG \eqref{eq in vec u} on an interval $I$ such that for some $\e\in(0,\e_*]$, $R\in(2\e,R_*]$ and $\t_1<\t_2\in I$, 
\EQ{
 E( u) < J(Q) + \e^2, \pq d_Q( u(\t_1))<R<R+R^2<d_Q( u(\t_2)),}
then for all $t\in(\t_2,\I)\cap I=:I'$, we have $d_Q( u(t))>R$. 
\end{thm}
Moreover, there exist a constant $\de_*\in(0,\de_S)$ (independent of $u$) such that $\e_*<\e_0(\de_*)$, and disjoint subintervals $I_1,I_2,\dots\subset I'$ with the following property: On each $I_m$, there exists $t_m\in I_m$ such that 
\EQ{
 d_Q( u(t))\simeq e^{k|t-t_m|}d_Q( u(t_m)), \pq \min_{s=0,2}\sg K_s(u_1(t))\gec d_Q( u(t))-C_*d_Q( u(t)).}
where $\sg=\Sg( u(t))\in\{\pm 1\}$ is constant, $d_Q( u(t))$ is increasing for $t-t_m\gg R^2$, decreasing for $t_m-t\gg R^2$, and equals to $\de_X$ on $\p I_m$. For each $t\in I'\setminus\Cu_m I_m$ and $s=0,2$, we have $(t-1,t+1)\subset I'$, $d_Q( u(t))\ge \de_*$, and 
\EQ{
 \int_{t-1}^{t+1}\min_{s=0,2}\sg K_s(u_1(t'))dt' \gg R_*^2.}

\section{Dynamics after ejection}
\label{sec:dyn}

After the ejection, we obtain the blowup in the region $\Sg=-1$ by the same argument (Payne-Sattinger) as in the radial case \cite{NLKGrad}. 

For the scattering after ejection in the region $\Sg=+1$, we need some small modifications. First, we should replace the linear profile decomposition given in \cite{NLKGrad} with the general version in \cite{ScatBlow}. We can keep the Strichartz norm $L^3_tL^6_x$, but it is more convenient to use the symmetric $L^4_{t,x}$ norm. It is standard and easy to see that these norms are interchangeable for the solution $u$ of \eqref{eq in vec u}, because by H\"older's inequality 
\EQ{
 \pt \|u_1\|_{L^3_tL^6_x} \lec \|u_1\|_{L^4_tL^4_x}^{\frac13}\|u_1\|_{L^{8/3}_tL^8_x}^{\frac23},
 \pq \|u_1\|_{L^4_tL^4_x} \lec \|u_1\|_{L^\I_t L^2_x}^{\frac14}\|u_1\|_{L^3_tL^6_x}^{\frac34},}
and by Strichartz 
\EQ{
 \|u_1-v_1\|_{(L^\I_t H^1_x\cap L^{8/3}_tL^8_x)(0,T)}
 \pn\lec \|u_1^3\|_{L^1_tL^2_x(0,T)}
 \pt\lec \|u_1\|_{L^3_tL^6_x(0,T)}^3
 \pr\lec \|u_1\|_{L^4_tL^4_x(0,T)}  \|u_1\|_{L^{8/3}_tL^8_x}^{2},}
where $v$ denotes the free solution with $ v(0)= u(0)$. Hence each of the two norms can control the other (with the aid of the uniform energy bound in the region 
$\Sg=+1$). 

Another issue is that the nonlinear profiles for a sequence of solutions in $\HL_0$ do not necessarily have vanishing momentum. 
Hence, after constructing  the single profile $u^0$ which is not scattering, we need to Lorentz transform it back into $\HL_0$. 
Here a crucial observation is that $\|u\|_{L^4_tL^4_x(0,\I)}<\I$ is preserved by any Lorentz transform. More precisely, we have the following result. 

\begin{lem} \label{global Lorentz}
Let $u$ be a finite energy solution of NLKG \eqref{eq:NLKG} on $(T,\I)$, and let $u'$ be a Lorentz transform of $u$. Then there exists $T'\in\R$ such that $u'$ extends to a finite energy solution on $(T',\I)$. Moreover, if $\|u'\|_{L^4_{t,x}(t>T')}<\I$ then $\|u\|_{L^4_{t,x}(T,\I)}<\I$. 
\end{lem}
\begin{proof}
Let $\chi\in C_0^\I(\R^3)$ be the cut-off function as before, and let $w$ be the solution of NLKG with $\vec w(T)=(1-\chi(x/S))\vec u(T)$. If $S\gg 1$ is large enough, then $w$ is global with $\|w\|_{L^4_{t,x}(\R^{1+3})}<\I$, by the small data scattering theory. By the finite propagation property, $u(t,x)=w(t,x)$ for $|x|>S+|t-T|$. Hence $u$ extends to this region for $t<T$. The image of the region 
\EQ{
 (T,\I)\times\R^3 \cup\{(t,x)\mid|x|>S+|t-T|\}}
by any Lorentz transform contains $(T',\I)\times\R^3$ for some $T'>0$. Then the transform of $u$ solves NLKG for $t>T'$. 
After a suitable rotation, we may assume that the Lorentz transform is in the form \eqref{Lorentz}. 
If $\|u'\|_{L^4_{t,x}(t>T')}<\I$, then 
\EQ{
 \I>\|u'\|_{L^4_{t,x}(t>T')} \ge  \|u\|_{L^4_{t,x}(t\cosh\nu-x_1\sin\nu>T'\text{ and } t>T)}.}
Since the remaining region
\EQ{
 \{(t,x)\mid t>T,\ t\cosh\nu-x_1\sinh\nu<T',\text{ and }\ |x|<S+|t-T|\} }
is bounded in space-time, the $L^4_{t,x}$ norm of $u$ in that region is bounded by the Sobolev embedding. Hence $\|u\|_{L^4_{t,x}(t>T)}<\I$. 
\end{proof}

Similarly, we have a local version of the above:

\begin{lem} \label{loc Lorentz}
Let $u$ be a finite energy solution of NLKG \eqref{eq:NLKG} on a time interval $I\ni T$. Then there is an open neighborhood 
$O$ of the identity in the Lorentz group, such that the transform of $u$ by any $g\in O$ extends to a solution in a space-time region including a time slab which contains $T$. 
\end{lem}
\begin{proof}
Let $w$ be the global solution as given in the proof of the previous lemma. Then $u$ extends to $I\times\R^3\cup\{|x|>S+|t-T|\}$, which is mapped to a region containing a time slab by any Lorentz transform sufficiently close to the identity. 
\end{proof}

Therefore the argument in \cite{NLKGrad} works in the nonradial setting, by using the Lorentz transform of $u^0$ with $0$ momentum in applying the ejection lemma, and 
going back to the original profile $u^0$ when using it to approximate the minimizing sequence of solutions. 
Thus we obtain a critical element as well as its precompactness. After that the argument is the same as in Kenig-Merle \cite{KM2}, or more verbatim 
in~\cite{ScatBlow} (neither using the radial symmetry). 

Hence,  we arrive at the following conclusion. 

\begin{thm}
Let $0<\e\le\e_*$ and let $u\in C([0,T);\HL)$ be a solution of NLKG \eqref{eq in vec u} on a forward maximal interval $[0,T)$ such that $P( u)=0$, 
\EQ{
 E( u)\le J(Q)+\e^2, \pq d_Q( u(t))\ge R_*, \pq \Sg( u(t))=\pm 1 \ (0\le t<T).}
If $\Sg=-1$, then $T<\I$. If $\Sg=+1$, then $T=\I$ and $u$ scatters to $0$ as $t\to\I$. Moreover, there is $M\in(0,\I)$ determined only by $\e$ such that $\|u_1\|_{L^4_{t,x}(t>0)}\le M$. 
\end{thm}

\section{Global dynamics}

Now we investigate some basic topological properties of the scattering and blow-up sets. The minimized energy by the Lorentz transform is given by 
\EQ{
 E_m(\fy) = \sqrt{|E(\fy)^2-|P(\fy)|^2|}\sign( E(\fy)^{2} - P(\fy)^{2} ).}
For any $\e\ge 0$, we define 
\EQ{
 \pt \HL^{<\e} = \{\fy \in L^2\mid E_m(\fy)<J(Q)+\e^2\},
 \pr \HL^{=\e} = \{\fy \in L^2\mid E_m(\fy)=J(Q)+\e^2\},
 \pr \HL^{\le\e} = \{\fy \in L^2\mid E_m(\fy)\le J(Q)+\e^2\}.}
For $\s=\pm$ and any $*$, we define 
\EQ{
 \pt \S_\s^* = \{ u(0) \in \HL^* \mid u(t) \text{ scatters as $\s t\to\I$}\},
 \pr \B_\s^* = \{ u(0) \in \HL^* \mid u(t) \text{ blows up in $\s t>0$}\},
 \pr \T_\s^* = \HL^* \setminus(\S_\s^* \cup \B_\s^*).}
The definition of the ``trapped set" $\T_\s^*$ appears rather awkward. It follows from the preceding that any solution in $\T_+^{<\e}$ is forward global, and after the energy-minimizing Lorentz transform, it stays close to $\{\pm Q(x-c)\}_{c\in\R^3}$ within distance $2\e$ for large $t$, provided that $\e\le\e_*$. 

It is obvious that $\S_\s\cap\B_\s=\emptyset$ and $X_-^* = \{\bar{\fy} \mid \fy \in X_+^*\}$ (where $X$ stands for any of the three types of sets). The scattering theory implies that $\S_\s^{<\e}$ is open for any $\e\ge 0$. 

To see that $\B_\e^{<\e}$ is open for $\e\le\e_*$, let $u$ be a solution on $[0,T_*)$ with $P(u)=0$ blowing up at $T_*$, and we proceed as in the radial case. The local Cauchy theory implies $\| u(t)\|_{L^2}\gec|T^*-t|^{-1/2}$, and so from the identity
\EQ{
 \p_t^2\|u_1(t)\|_{L^2_x}^2=2[\|u_2(t)\|_{L^2_x}^2-K_0(u_1(t))] \ge 6\|u_2\|_{L^2_x}^2 + 2\|u_1\|_{H^1_x}^2 - 8E(u),}
we deduce that $\p_t\|u_1\|_{L^2_x}^2=2\LR{u_1|u_2}\to\I$ and $K_0(u_1)\to-\I$ as $t\to T_*-0$. In other words, for any $M\in(0,\I)$ there exists $T\in(0,T_*)$ such that $\LR{u_1|u_2}\ge M$ and $K_0(u_1)\le-M$ for $T\le t<T_*$. 
Let $v$ be a solution with $ v(T)$ close to $ u(T)$ in $\HL$, and let $w$ be its Lorentz transform with $P(w)=0$. Since $P(v)\simeq P(u)=0$, the transform is close to the identity, and so by Lemma \ref{loc Lorentz}, $w$ is a local solution around $t=T$, and moreover ${w}(T)$ is close to $ u(T)$ in $\HL$. In particular $\LR{w_1|w_2}(T)>M/2$ and $K_0(w_1(T))<-M/2$. By the above identity, $\LR{w_1|w_2}$ is increasing as long as $K_0(w_1)<0$, and the latter can be changed only if $d_Q( w)\le\de_*$, which is impossible if $\LR{w_1|w_2}>M/2\gg 1$. Hence $\LR{w_1|w_2}$ is increasing, as well as $\|w_1\|_{L^2}$, and so $w$ blows up in $t>T$ by Payne-Sattinger's argument. Then Lemma \ref{global Lorentz} implies that $v$ also blows up in $t>T$, which means that $\B_+^{<\e}$ is open in $\HL$. Hence $\T_+^{\le\e}$ is closed. 

We already know that the $9$ intersections of $(\S_+^{<\e},\B_+^{<\e},\T_+^{<\e})$ by $(\S_-^{<\e},\B_-^{<\e},\T_-^{<\e})$ are all non-empty containing infinitely many radial solutions, for any $\e>0$. 
The construction in \cite{NLKGrad} works even if we choose a nontrivial dispersive component $\|\ga(0)\|_{L^2}\ll\e$. In this fashion, we easily obtain nonradial elements in each set (besides those generated by the invariant transforms). 

$\S_+^{<\e}$, $\S_+^{=\e}$ and $\S_+^{\le\e}$ are connected for any $\e\ge 0$, by the same proof as in the radial case in \cite{NLKGrad}. 

\section{Center-stable/unstable manifolds}
\label{sec:CS}
In the rest of paper, we prove that the solutions staying forever close to the manifold of the ground states scatter to the manifold, using the dispersive estimate for the linearized operator. 
Here we encounter more difference from the radial case, due to the parameter freedom. Indeed, the argument is more similar to the NLS case in \cite{NakS2}. 
\subsection{Equations}
First we need more detailed analysis of the Lorentz invariance. With any $p\in\R^3$, we associate the following Lorentz transform 
\EQ{
 \pt p=:s\th, \pq \th\in S^2, \pq s=\sinh\nu\ge 0, \pq c:=\sqrt{1+|p|^2}=\cosh\nu, 
 \pr \t:=\frac{p}{\sqrt{1+|p|^2}}=\th\tanh\nu, \pq x_\th:=\th(\th\cdot x), \pq x_\perp=x-x_\th, 
 \pr u_p(t,x):=u(ct - sx\cdot\th, c(x_\th - \t t)+x_\perp)
 =u(tc-p\cdot x,x+(c-1)x_\th-tp).}
In particular, the static solution $Q(x)$ is transformed into the traveling wave solution 
\EQ{
 Q(c(x_\th-\t t)+x_\perp)=Q(x-\t t+(c-1)(x-\t t)_\th)}
with the velocity $\t=p/\LR{p}$, for each fixed $p\in\R^3$. Its trajectory is in the 6-dimensional manifold in $H^1$ consisting of 
\EQ{
 Q(p,q)(x):=Q(c(x-q)_{\theta}+(x-q)_\perp),}
whose vector version is denoted by 
\EQ{\label{eq:VQdef}
  \VQ(p,q) := (\D+i\t\cdot\na) Q(p,q).}
The NLKG equation for traveling waves is transformed into the following equation of $\VQ(p,q)$:
\EQ{\label{eq:vecQeq}
 i\D\VQ +  \t\cdot\na \VQ = i\VQ_1^3}
  where $\VQ(p,q)_{1}^{3}=[\VQ(p,q)_{1}]^{3}= Q(p,q)^{3}$ (and similarly for other powers). 
Differentiating \eqref{eq:vecQeq} in $(p,q)$, we obtain the linearized equations 
\EQ{\label{eq:lineqs}
 \pt (i\D\LL+\t\cdot\na)\p_q\VQ=0, \pq \p_q\VQ=-\na\VQ, 
 \pr (i\D\LL+\t\cdot\na)\p_p\VQ=(\p_p\t)\cdot(\p_q\VQ),}
where the $\R$-linear self-adjoint operator $\LL$ now depends on the parameters:
\EQ{
 \LL=\LL(p,q) = 1 - 3\D^{-1}\VQ(p,q)_1^2\D^{-1}\Re.}
 Note that this agrees with our previous definition~\eqref{eq:LNdef} for $p=q=0$. 
Decomposing the solution $u$ in the form 
\EQ{
  u =  \VQ(p,q) +  v(x-q),}
we derive the equation for the perturbation $v$, i.e.,  
\EQ{
  v_t = (i\D\LL_ p+\t(p)\cdot\na) v  + (\dot q-\t(p))\cdot\na(\VQ_p+ v) - \dot p\cdot\p_p\VQ_p - iN_p(v_1),}
where $\VQ_p:=\VQ(p,0)$, $\LL_ p:=\LL(p,0)$ and $N_p(v_1):=3\VQ_{p1}v_1^2+v_1^3$. For brevity, let 
\EQ{
 \A_p := i\D\LL_ p + \t(p)\cdot\na, \pq \ga:=q-\int_0^t\t(p(s))\, ds, \quad \pi:=(p,\ga)}
then the equation can be  rewritten in the form 
\EQ{
 v_t = \A_p v  + \dot \ga\cdot\na(\VQ_p+v)-\dot p\cdot\p_p\VQ_p - iN_p(v_1).}
The natural orthogonality condition is 
\EQ{\label{eq:omorth2}
 0=\om( v,\p_q \VQ )=\om( v,\p_p \VQ ).}
 Here, and in what follows, it will be understood that all derivatives of $\VQ$ are to be evaluated at $q=0$. 
 In particular, \eqref{eq:VQdef} implies that 
\EQ{
 \p_{p}\VQ (0,0) (x) = i\nabla Q(x),\quad \p_{q}\VQ(0,0)(x)=-\na\VQ(0,0)(x)=-\na\D Q(x),}
which constitute the root modes for $i\D\LL$, see \eqref{eq:geneifunc}. 
By differentiation of \eqref{eq:omorth2}, we obtain the parameter evolution 
\EQ{ \label{mod eq}
 0\pt=\om( v_t,\p_\al \VQ )+\om( v,\p_\al\p_p \VQ )\dot p
 \pr=-\om( v,\A_p\p_\al \VQ)+\om(\dot\ga\na( \VQ + v)-\dot p\p_p \VQ ,\p_\al \VQ )
 \prQQ-\om(iN_p(v_1),\p_\al \VQ )+\om( v,\p_\al\p_p \VQ )\dot p,}
for $\al=q_1,q_{2},q_3,p_1,p_{2},p_3$. In view of~\eqref{eq:lineqs} and~\eqref{eq:omorth2} the first term of~\eqref{mod eq} vanishes, whence
\EQ{
 & \dot\ga[\om(\p_q \VQ ,\p_\al \VQ )-\om( v,\p_q\p_\al \VQ )]
  +\dot p[\om(\p_p \VQ ,\p_\al \VQ )-\om( v,\p_p\p_\al \VQ )] 
 \pr=-\om(iN_p(v_1),\p_\al \VQ ),}
which implies 
\EQ{
 |\dot \pi| \lec \|e^{-|x|/4}v_1\|_{2}^2,}
provided that the $6\times 6$ matrix $\om(\p_\al \VQ_{p},\p_\be \VQ_{p})$ is non-degenerate. Certainly it is at $p=0$, and therefore
remains so  for small $p$ by continuity. 

Now we look for a bounded global solution $v$ by a contraction argument, for a given initial data near $\VQ(0,0)$ in the center-stable direction. The contraction mapping $(v,\pi)\mapsto(\nx v,\nx \pi)$ is defined by 
\EQ{\label{eq:CSsys}
 \pt (\p_t-\A_p)\nx v - \nx\ga_t \cdot\na(\VQ_p+\nx v) + \nx p_t\cdot\p_p\VQ_p = - iN_p(v_1),
 \pr \p_t\om(\nx v,\p_\al\VQ_p)=0,\quad(\forall\al=p_1,p_2,p_3,q_1,q_2,q_3)}
with the initial constraint
\EQ{\label{eq:CSsysinit}
 0=\om(\nx v(0),\p_\al\VQ (0,0))\:\:(\forall\al),\;\;\;  \pq \pi(0)=\nx\pi(0)=0,\ \nx\ga(0)=0.}
 In \eqref{eq:CSsys}, \eqref{eq:CSsysinit}, the soliton~$\VQ(p,q)$ is to be evaluated at $q=0$ and $p(t)$ which is determined by 
 the {\em given} path $\pi(t)$.  
The orthogonality equation is equivalent to 
\EQ{ \label{eq pi}
 \pt\nx\ga_t \om(\p_q \VQ_p,\p_\al \VQ_p)+ \nx p_t \om(\p_p \VQ_p,\p_\al \VQ_p)
 \pr= \ga_t \om(\nx v,\p_q\p_\al \VQ_p) + p_t \om(\nx v,\p_p\p_\al \VQ_p) 
 -\om(iN_p(v_1),\p_\al \VQ).}
We apply the symplectic decomposition at $\pi(0)=0$
\EQ{
 1=P_0+P_1=P_0+P_++P_-+P_c, \pq P_\pm\nx v=:\nx\la_\pm g_\pm, \; \pq P_c\nx v=\nx z,}
where $P_0$, $P_\pm$ and $P_c$ are projections onto the subspaces spanned respectively by $\{\nabla\D Q,i\na Q\}$, $g_\pm$, and the rest, cf.~\eqref{eq:geneifunc}.
The equation of $\nx v$ yields 
\EQ{\label{eq:keysys}
 \pt (\p_t\mp k)\nx\la_\pm=\om(a\cdot\na \nx v+f,\pm g_\mp), \pq a:=\t(p)+\nx\ga_t, 
 \pr (\p_t-\A_0)\nx z= P_c(a\cdot\na \nx v+ f), 
 \pr f:=-3i(\VQ_{p1}^2-\VQ_{01}^2)\nx v_1 + \nx\ga_t\cdot\na\VQ_p-\nx p_t\cdot\p_p\VQ_p-iN_p(v_1).}
 Note that $k, g_{\pm}$, $\A_{0}, P_{c}$ are all time-independent. 
For any given $\la_-(0)$, the unique bounded solution of $\nx\la_{\pm}$ is given by 
\EQ{ \label{eq la}
 \nx\la_+(t) &=  - \int_t^\I e^{k(t-s)}\om(f(s),g_-)\, ds, \\
  \pq \nx\la_-(t) &= e^{-kt}\la_-(0)-\int_0^te^{k(s-t)}\om(f(s),g_+)\, ds,}
while the equation for $\nx z$ is rewritten as 
\EQ{ \label{eq z}
 (\p_t-\A_0)\nx z = P_c(a\cdot\na\nx z+g), \pq g:=a\cdot\na P_0\nx v+f.}

\subsection{A priori bounds} \label{cst bd}
Inverting the $6\times 6$ matrix in \eqref{eq pi}, we obtain the estimate 
\EQ{
 \|\nx \pi_t\|_{(L^1\cap L^\I)_t}
 \lec \|\pi_t\|_{L^1\cap L^\I}\|\nx v\|_{L^\I L^2} 
 + \|N_p(v_1)\|_{(L^1\cap L^\I)_t(L^1+L^\I)_x}.}
Since $P_0(p(t))\nx v(t)=0$, we have, in any Sobolev space $X$ on $\R^3$, 
\EQ{
 \|P_0\nx v\|_X \lec \|p\|_{L^\I_t}\|\nx v\|_X \ll \|\nx v\|_X,}
and so 
\EQ{
 \|\nx v\|_X \lec \|P_1\nx v\|_X \lec |\la|\|g\|_X + \|\nx z\|_X.}
Applying Young's inequality  to \eqref{eq la} yields 
\EQ{
 \|\nx\la\|_{(L^2\cap L^\I)_t} \lec |\la_-(0)|+\|f\|_{(L^2\cap L^\I)_t(L^1+L^\I)_x}}
and  applying Proposition~\ref{prop:Stz ga} to \eqref{eq z} (see Section \ref{ss:defs in Sect 9} for the definition), 
\EQ{
 \pt \|\nx z\|_{L^\I_t L^2_x \cap L^2_t X} \lec \|z(0)\|_2 + \|g\|_{L^1_tL^2_x+L^2_t Y},
 \pr X:=B^{-5/6}_{6,2}\cap \D^{\nu/2}L^{2,-\s}, \pq Y:=B^{5/6}_{6/5,2}\cap \D^{-\nu/2}L^{2,\s},} 
for any $\nu,\s>0$ satisfying the condition in the proposition, 
under the assumption 
\EQ{
 \|a\|_{L^\I}\lec\|p\|_{L^\I} \| \nx v\|_{X}+\|\nx\ga_t\|_{L^\I}\lec \de\ll 1.} 
The nonlinear terms are estimated by H\"older's inequality as follows: 
\EQ{
 \pt\|N_p(v_1)\|_{(L^1\cap L^\I)_t(L^1+L^\I)_x} \lec \|v\|_{Stz}^2,
 \pr\|f-N_p(v_1)\|_{(L^2\cap L^\I)_t(L^1+L^\I)_x} 
  \lec \|p\|_{L^\I}\|\nx v\|_{Stz}+\|\nx\pi_t\|_{L^1\cap L^\I},
 \pr\|g\|_{L^1_tL^2_x+L^2Y} \lec \|a\|_{L^\I}\|\nx v\|_{Stz}+\|p\|_{L^\I}\|\nx v\|_{L^2_t\D^{\nu/2}L^{2,-\s}}
 \prQQ \qquad+\|\nx\pi_t\|_{L^1\cap L^\I}+\|v\|_{Stz}^2,}
where we neglected higher order terms. Under the assumption
\EQ{ \label{contraction}
 |\la_-(0)|+\|z(0)\|_2 \le \de\ll 1, \pq \|\pi_t\|_{L^1\cap L^\I}+\|v\|_{Stz}\le B\de,}
for some big fixed $B>1$ and sufficiently small $\de>0$, we therefore obtain 
\EQ{
 \pt\|\nx\pi_t\|_{L^1\cap L^\I} \lec (B\de)^2 \ll \de,
 \pr\|\nx v\|_{L^\I_t L^2_x \cap L^2_t X} \lec \de + (B\de)^2 \ll B\de,}
by a bootstrap argument. Moreover, those global bounds imply that, as $t\to\I$, $\nx\pi$ converges, $\nx\la_\pm\to 0$, and $\nx z(t,x-b(t))$ scatters, where 
\EQ{ \label{def b}
 b(t):=\int_0^t a(s)\, ds.}
Hence $P_0\nx v(t)\to 0$ strongly, and so $\nx v(t,x-b(t))$ also scatters, 
where $\dot b=a=\t(p)+\ti\ga_t\to\t(p_\I)$ converges, although we cannot generally 
approximate it by a (Lorentz transform of a) free solution, since $b(t)-t\t(p_\I)$ does not necessarily converge. 

\subsection{Difference estimate} \label{diff est}
The above argument does not apply to the difference of two solutions, 
due to the term $a\cdot\na\nx z$ in the $\nx z$ equation, which causes two problems: 
possible growth in time and derivative loss for the difference. Hence we employ a 
rather weak norm for the difference, with an exponential weight in $t$ and values in $H^{-1}_x$. 
This is still sufficient, because the main issue in the difference estimate is 
 to control $\la_+(0)$ in the contraction argument. To be more precise, fix $\ro>0$ such that 
\EQ{
 0<\de^{1/4} \ll \ro \ll \min(1,k),}
and define the function space $G$ on $(0,\I)$ by the norm 
\EQ{
 \|f\|_G:=\sup_{t>0}e^{-\ro t}|f(t)|.}
We estimate the difference of the mapping of two given  $(v^j,\pi^j)$ ($j=0,1$), denoted by $\diff X=X^1-X^0$, in the following norm 
\EQ{
 \|\diff{\nx\pi}_t\|_G + \|\diff{\nx\la}\|_G + \|\diff{\nx z}\|_{G_tH^{-1}_x}.}
The difference of \eqref{eq pi} yields
\EQ{
 \|\diff{\nx\pi}_t\|_G \pt\lec \de\|\|\diff{\nx v}\|_{H^{-1}}+|\diff{\pi}|+|\diff{\pi}_t|+|\diff{\nx\pi}_t|+\|\diff v\|_{H^{-1}} \|_G 
 \pr\lec \de\|\diff{\nx v}\|_{G_tH^{-1}_x}+\de\ro^{-1}\|\diff{\nx\pi}_t\|_G+\de\ro^{-2}\|\diff{\pi}_t\|_G+\de\ro^{-1}\|\diff v\|_{G_tH^{-1}_x}.}
The difference of \eqref{eq la} gives us via Young's inequality 
\EQ{
 \|\diff{\nx\la}\|_G \pt\lec |\diff\la_-(0)|+\frac{\de}{k}[\|\diff{\nx v}\|_{G_tH^{-1}_x}+\|\diff{\nx\pi}_t\|_G+\ro^{-1}\|\diff{\pi}_t\|_G+\|\diff v\|_{G_tH^{-1}_x}].}
For $\diff{z}$, we need a change of variables to avoid a derivative loss due to the transport term. Let $b(t)$ be as in \eqref{def b}, and  define 
\EQ{
 \t_b\fy(x):=\fy(x-b(t)), 
 \pq \z(t)=\t_b z(t), \pq \nx\z(t)=\t_b \nx z(t), \pq T_bA=\t_b A \t_b^{-1},}
where $A$ is any operator on $L^2(\R^3)$. Then \eqref{eq z} is transformed to 
\EQ{ \label{eq ze}	
 (\p_t-\A_0)\nx\z = [3i(T_bQ^2-Q^2)-T_bP_d a\cdot\na]\nx\z + \t_bP_cg.}
Now we employ the linearized energy of regularity level $H^{-1}$:
\EQ{
 \|\y\|_{E^{-1}}^2\pt:=\LR{\LL(i\D\LL)^{-1}P_c\y|(i\D\LL)^{-1}P_c\y}
 \simeq\|(i\D\LL)^{-1}P_c\y\|_2^2\simeq\|P_c\y\|_{H^{-1}}^2,
 \\ \p_t\|\y\|_{E^{-1}}^2&=2\LR{\LL(i\D\LL)^{-1}P_c(\p_t-i\D\LL)\y|(i\D\LL)^{-1}P_c\y}
 \pr\lec \|(\p_t-i\D\LL)P_c\y\|_{H^{-1}}\|\y\|_{E^{-1}}.}
From the difference of \eqref{eq ze} we infer that 
\EQ{
 \p_t\|\diff{\nx\z}\|_{E^{-1}}\lec\de[|\diff b|+|\diff\pi|+|\diff{\nx\pi_t}|+\|\diff{\nx\z}\|_{H^{-1}}+\|\diff v\|_{H^{-1}}],}
and integrating it, 
\EQ{
 \|\diff{\nx\z}\|_{G_tH^{-1}_x} \lec \|\diff z(0)\|_2+\de[\ro^{-2}\|\diff\pi_t\|_G+\ro^{-3}\|\diff{\nx\pi_t}\|_G+\ro^{-1}\|\diff v\|_{G_tH^{-1}_x}].}
The difference before the translation is bounded by 
\EQ{
 \|\diff{\nx z}\|_{H^{-1}_x}-\|\diff{\nx\z}\|_{H^{-1}_x}
 \pt\le\|(\t_{b^0}-\t_{b^1})z^0\|_{H^{-1}_x} \lec|\diff b|\|z^0\|_{L^2_x}
 \lec\de|\diff b|,}
and so 
\EQ{
 \|\diff{\nx z}\|_{G_tH^{-1}_x} \lec \|\diff{\nx\z}\|_{G_tH^{-1}_x}+\de[\ro^{-2}\|\diff\pi_t\|_G+\ro^{-3}\|\diff{\nx\pi_t}\|_G].}
Thus we obtain
\EQ{
 \pt\|\diff{\nx\pi}_t\|_G+\|\diff{\nx v}\|_{G_tH^{-1}_x} 
 \pn\lec |\diff{\la_-(0)}|+\|\diff{z(0)}\|_2 + \de\ro^{-3}\|\diff{\pi}_t\|_G+\de\ro^{-1}\|\diff{v}\|_{G_tH^{-1}_x}.}
In conclusion, there is a unique fixed point $(v,\pi)$ in the set \eqref{contraction}, where we have 
\EQ{
 \dot b = a = \t(p)+\dot\ga=\dot q.}
It depends continuously on the initial data $(\la_-(0),z(0))\in\R\times L^2_x$ in the above topology. 
The same estimate holds for the derivatives with respect to the data, which implies that $v(0)$ of the fixed point is smoothly parametrized by $(\la_-(0),z(0))\in \R\times P_c(L^2)$. 
We have already seen in Section~\ref{cst bd} that, as $t\to\I$, $\la(t)\to 0$, $\pi(t)\to\exists\pi_\I$, $\dot b(t)=a(t)=\dot q(t)\to\t(p_\I)$, and 
\EQ{
 v(t,x-b(t)+b(0)-q(0)) = u(t,x) - \VQ(p(t),q(t))(x)}
scatters. This means that, at least in a weaker or localized topology, the solution $u$ converges to the family of ground states. Moreover, we have 
\EQ{ \label{EP asy}
 \pt E(u)=E(\VQ_{p_\I})+\|v_\I\|_2^2/2 = J(Q)\LR{p_\I}+\|v_\I\|_2^2/2, 
 \pr P(u)=P(\VQ_{p_\I})+P(v_\I) = J(Q)p_\I+P(v_\I),}
where $v_\I:=\lim_{t\to\I}e^{-i\D t}[u(t)-\VQ(p(t),q(t))]$. 

Note that we started with a fixed parameter at $t=0$, normalized to $(0,0)$ by a Lorentz transform, and ended up with some non-zero but small parameter $\pi_\I$ at $t=\I$. It is in general difficult to reverse this process starting from $t=\I$ because of the growth in the modulation parameter, unless working with localized dispersive data, which would yield better asymptotic control on the parameter. 

\subsection{Uniqueness}
Next we prove that any solution which stays forever close to the family of ground states is necessarily a Lorentz transform of one of those 
constructed above. Thus,  let $u$ be a solution of NLKG satisfying 
\EQ{ \label{stay sol}
 \sup_{t\ge 0} \inf_{q\in\R^3}\|u(t)-\VQ(x-q)\|_2\lec\de.}
First notice that this property is preserved by Lorentz transforms which are  $O(\de)$ close to the identity, provided
 $\de>0$ is small enough. This is because the above condition implies that for some $R(\de)>0$ and at each $t>0$ there is a ball of radius $R$ in $\R^3$ such that the linear energy of $u$ is at most $O(\de^2)$ in the exterior of the ball. Then by the same argument as in Lemma~\ref{global Lorentz}, the solution extends at least to the exterior of the light cones emanating from this ball, with a uniform energy bound of $O(\de^2)$, while the interior energy can be controlled by the linear energy inequality for time $O(\de)$, whence the claim.  

Next, the implicit function theorem implies that there is a unique $(p,q)\in\R^6$ such that for $u_p(0,x-q)=\VQ+v(0,x)$, 
\EQ{
 |p|+\|v(0)\|_2 \lec \de, \pq 0=\om(v(0),\p_\al\VQ).}
Hence replacing $u$ by $u_p(t,x-q)$, we may assume, in addition to \eqref{stay sol}, that 
\EQ{
 u(0)=\VQ+v(0), \pq \om(v(0),\p_\al\VQ)=0.}
Let $\pi(t)=(p(t),\ga(t))$ and $u(t)=\VQ(p(t),q(t))+v(t,x-q(t))$ with $q(t)=\ga(t)+\int_0^t\t(p(s))\, ds$ where $\pi$ is evolved by \eqref{mod eq}, which is locally uniquely solvable as long as $|\pi(t)|+\|v(t)\|_2\ll 1$, preserving the orthogonality $\om(v(t),\p_\al\VQ(p(t),q(t)))=0$. Meanwhile, \eqref{stay sol} implies, via the implicit function theorem, that at each $t$ there is a unique $(p(t),q(t))\in\R^6$ of size $O(\de)$ such that the orthogonality property holds. This gives an a priori bound of
size $O(\de)$ on $\pi(t)$ solving \eqref{mod eq}, so it extends to all $t\ge 0$, satisfying  
\EQ{ \label{cond on mfd}
 \|\pi\|_{L^\I(0,\I)} + \|v\|_{L^\I_t(0,\I;L^2_x)} \lec \de, \pq 0=\om(v(t),\p_\al\VQ(p(t),q(t))).}
 
To establish  the global Strichartz bound, consider the local version for any $T>0$
\EQ{
 N_T:=\|\pi_t\|_{L^1_t(0,T)\cap L^\I_t(T,\I)} + \|v\|_{L^2_t(0,T;X)\cap L^\I_t H^1_x}.}
Since $(\pi,v)$ is bounded, we have \eqref{eq pi}, \eqref{eq la} and \eqref{eq z} with $(\nx\pi,\nx\la,\nx z,\nx v)=(\pi,\la,z,v)$. 
Then by the same argument as in Section~\ref{cst bd} together with~\eqref{cond on mfd}, we obtain 
\EQ{
 N_0 \lec \de, \pq N_T \lec \de+\de N_T + N_T^2,}
uniformly in $T>0$. Hence by the dominated convergence theorem, we conclude that $N_\I\lec\de$. 
Now that $(\pi,v)$ is a solution belonging to the set \eqref{contraction}, the uniqueness follows from the 
contraction principle via the difference estimates of~Section \ref{diff est}. 

\subsection{Threshold solutions}
Finally, we prove the Duyckaerts-Merle type classification of the solutions with the 
threshold energy $E(u)^2-|P(u)|^2=J(Q)^2$, that is, there are only three solutions modulo the symmetries. 

Let $u$ be a solution with $E(u)^2-|P(u)|^2=J(Q)^2$, satisfying \eqref{stay sol}. 
Then after a Lorentz transform, it is given by the above construction, and \eqref{EP asy} 
implies that $v_\I=0$ if $E(u)^2-|P(u)|^2=J(Q)^2$. Hence $z(t)=P_cv(t)\to 0$ strongly in $L^2_x$ as $t\to\I$, which 
allows us to apply the Strichartz estimate, Proposition~\ref{prop:Stz ga}, to $z$ starting from $t=\I$. Consider a weight function 
\EQ{
 \nu_T(t) = \min(e^{\mu(t-T)},1),}
for any fixed $\mu\in(k/2,k)$ and $T\to\I$. Then by the same argument as in Section~\ref{cst bd}, we obtain 
\EQ{
 \|\nu_T\pi_t\|_{L^1\cap L^\I} + \|\nu_T v\|_{L^\I_t H^1_x \cap L^2_t X} \lec e^{-\mu T}\de,}
where the right-hand side comes from $\|\nu_Te^{-\mu t}\la_-(0)\|_{L^1\cap L^\I}$. Taking $T\to\I$, we deduce that 
\EQ{ \label{exp dec}
 \|e^{\mu t}\pi_t\|_{L^1\cap L^\I} + \|e^{\mu t}v\|_{L^\I_t H^1_x \cap L^2_t X}\lec \de,}
which allows us to bound the difference of any two such solutions from $t=\I$. Let $(v^j,\pi^j)$ $(j=0,1)$ be two solutions enjoying the above exponential decay. Then by the same argument as in Section~\ref{diff est}, we obtain from the difference of \eqref{eq pi} 
\EQ{
 \|\diff\pi_t\|_{L^\I_t} \pt\lec \de\|e^{-\mu t}[\|\diff v\|_{H^{-1}_x}+|\diff\pi_t|+|\diff\pi|]\|_{L^\I_t}
 \pr\lec \de[\|\diff v\|_{L^\I_t H^{-1}_x}+\|\diff\pi_t\|_{L^\I_t}],}
where $\de$ comes from the norm \eqref{exp dec} for both $(v^j,\pi^j)$. Similarly, we obtain from the difference of~\eqref{eq la}
\EQ{
 \|\diff\la\|_{L^\I_t} \lec |\diff\la_-(0)|+\de[\|\diff\pi_t\|_{L^\I_t}+\|\diff v\|_{L^\I_t H^{-1}_x}].}
From the difference of \eqref{eq ze}, we have the energy inequality 
\EQ{
 -\p_t\|\diff\z\|_{E^{-1}}\lec \de e^{-\mu t}[|\diff b|+|\diff\pi|+|\diff\pi_t|+\|\diff\z\|_{H^{-1}_x}].}
Since $\|\z^j(t)\|_2\to 0$ as $t\to\I$, we can integrate the above from $t=\I$, which yields 
\EQ{
 \|\diff\z\|_{L^\I_t H^{-1}_x} \lec \de[\|\diff\pi_t\|_{L^\I_t}+\|\diff z\|_{L^\I_t H^{-1}_x}].}
The difference before the transport is estimated similarly 
\EQ{
 \|\diff z\|_{L^\I_t H^{-1}_x}- \|\diff\z\|_{L^\I_tH^{-1}_x} \lec \|e^{-\mu t}\diff b \|e^{\mu t}z^0\|_{L^2_x}\|_{L^\I_t} \lec \de\|\diff\pi_t\|_{L^\I_t}.}
Thus we obtain 
\EQ{
 \|\diff\pi_t\|_{L^\I_t} + \|\diff v\|_{L^\I_tH^{-1}_x} \lec |\diff\la_-(0)|,}
which implies that any solution satisfying \eqref{cond on mfd} and $E(u)^2-|P(u)|^2=J(Q)^2$ is uniquely determined by $\la_-(0)\in\R$ of $O(\de)$. Moreover, these conditions are invariant under the  forward time shift, and $\la_-(t)\to 0$ as $t\to\I$. 

Therefore, all such solutions are classified into three distinct cases: $\la_-(t)>0$ decreasing to $0$, $\la_-(t)\equiv 0$, and $\la_-(t)<0$ increasing to $0$. Moreover, solutions in each class are translations of the others in the same class. The solution with $\la_-(t)=0$ is the exact ground state $u=Q$, and applying the one-pass theorem backward in time, we deduce that the solutions with $\la_-(t)>0$ blow up in finite time $t<0$, while the solutions with $\la_-(t)<0$ scatter as $t\to-\I$.

\section{Linearized dispersive estimate}

This section is devoted to establishing the dispersive bounds needed in the construction of the center-stable manifold.
As in Beceanu's theorem~\cite{Bec2} on energy class (or better) center-stable manifolds for the cubic Schr\"odinger equation in~$\R^{3}$ for general data,
we are lead here to a linear equation involving a small perturbation in the form $ia(t)\cdot\nabla$, which reflects the translation invariance. 
A general approach covering this case was developed by Beceanu in~\cite{Bec1} for the Schr\"odinger case in~$\R^{3}$, and we follow his method in a wide sense.  However,  the technical details
here are quite different from those of~\cite{Bec1}. For a related, but much simpler, statement see also Appendix~B of~\cite{NakS2}. 

\subsection{Linear matrix operator} \label{ss:defs in Sect 9}
In order to use the Kato theory, we need to make our operator $\C$-linear. Define a matrix operator $\LH$ on $L^2(\R^3;\C^2)$ by 
\EQ{\label{eq:HHmatrix} 
 \LH = \mat{\D-\frac 32Q^2\D^{-1} & -\frac 32Q^2\D^{-1} \\ \frac 32Q^2\D^{-1} & -\D+\frac 32Q^2\D^{-1}},}
such that $i\LH$ is an extension of $i\D\LL$ in the sense that 
\EQ{
 i\D\LL\fy = \psi \iff i\LH\mat{\fy \\ \bar{\fy}} = \mat{\psi \\ \bar{\psi}}.}
In the same way, the scalar operators $i$ and $\D\LL$ are extended to $i\s_3$ and $\s_3\LH$. Here 
$\s_{3}=\mat{ 1&0\\ 0&-1}$ as usual.  The inner product and the symplectic form $\om$ are extended 
in the following fashion: 
\EQ{
 \LR{u,v}:=\frac 12\int_{\R^3}\sum_{j=1}^2u_j(x)\bar{v_j(x)}dx, \pq \Om(u,v):=\LR{\s_3\D^{-1}u,v}.}
For any vector $u\in L^2(\R^3;\C^2)$, we define scalar functions $u^\pm$ by 
\EQ{
 u^+:=\D^{-1}(u_1+u_2), \pq u^-:=u_1-u_2.}
Then we have 
\EQ{\label{eq:split efunc}
  \LH\fy=\psi \iff L_+\fy^+=\psi^-, \pq \fy^-=\psi^+,}
which enables us to determine the spectrum of $\LH$ from that of $L_+$, cf.~\cite{DS}:
\EQ{
 \s(i\LH)=i(-\I,-1] \cup i[1,\I) \cup \{-k,k\} \cup \{0\},}
and the complete set of (generalized) eigenspaces of the nonselfadjoint operator $i\LH$ is spanned by 
\EQ{
 i\LH\WV{g_\pm}=\pm k\WV{g_\pm}, \pq i\LH\WV{\na Q}=0, \pq i\LH\WV{RQ}=-\WV{\na Q},}
where 
\EQ{
 \WV{g_\pm}:=\frac{1}{\sqrt{2k}}\mat{\D\ro\pm ik\ro \\ \D\ro\mp ik\ro}, \pq \WV{RQ}:=\mat{i\na Q \\ -i\na Q}, \pq \WV{\na Q}:=\mat{\D\na Q \\ \D\na Q}.}
These functions enjoy the following orthogonality properties: 
\EQ{
 \pt 0=\Om(\WV{g_\pm},\WV{RQ})=\Om(\WV{g_\pm},\WV{\na Q}),
 \pr 1=\Om(\WV{g_+},\WV{g_-}), \pq \Om(\WV{\na_j Q},\WV{R_kQ})=\de_{j,k}\|\p_1Q\|_2^2=\de_{j,k}J(Q).}
Hence the projection $P=P_c$ onto the continuous spectrum is given by 
\EQ{
 \pt P_c=1-P_d, \pq P_d=P_0+P_++P_-, \pq P_\pm u:=\Om(u,\pm\WV{g_\mp})\WV{g_\pm},
 \pr P_0u:=\sum_{j=1}^3J(Q)^{-1}[\Om(u,\WV{R_jQ})\WV{\na_jQ}-\Om(u,\WV{\na_jQ})\WV{R_jQ}]. }
The symplectic orthogonality is rewritten in the scalar form 
\EQ{ 
 \pt 0=\Om(u,\WV{g_+})=\Om(u,\WV{g_-}) \iff 0=\LR{\ro,u^-}=\LR{\ro,u^+},
 \pr 0=\Om(u,\WV{\na Q})=\Om(u,\WV{RQ}) \iff 0=\LR{\na Q,u^-}=\LR{\na Q,u^+},}
where $\LR{\cdot,\cdot}$ denotes the standard inner product on $L^2(\R^3;\C)$. 

\subsection{Dispersive estimates for the scalar operator}
First we recall the weighted $L^2$ estimates for the resolvent and the propagator of the scalar Klein-Gordon $\LR{\na}$ with a potential. 
All the estimates in this subsection are classical, but we provide proofs for the reader's convenience. 

Let $B^s_{p,q}(\R^{n})$ denote the inhomogeneous Besov space based on $L^p(\R^n)$ for any $n\ge 1$, $s\in\R$ and $p,q\in[1,\I]$. For brevity, we use the standard notation
\EQ{
 H^s := B^s_{2,2}, \pq C^s := B^s_{\I,\I}.}
The homogeneous versions are denoted by $\dot B^s_{p,q}$, $\dot H^s$ and $\dot C^s$, respectively. For $s\in(0,1)$, we have the equivalent semi-norms by the difference (cf.~\cite{BL})
\EQ{
 \pt\|\fy\|_{\dot B^s_{p,q}} \simeq \|\sup_{|y|\le\si}\|\fy(x)-\fy(x-y)\|_{L^p}\|_{L^q(d\si/\si)}.}
The weighted $L^2$ space $L^{2,s}(\R^{n})$ is defined by the norm
\EQ{
 \|\fy\|_{L^{2,s}} = \|\LR{x}^s\fy\|_{L^2(\R^n)}}
for any $s\in\R$. Hence $L^{2,s}$ is the Fourier image of $H^s$. 

We start with the resolvent estimate on the free part $\LH_0=\si_3\D$,  which is  a version of the classical   limiting absorption principle~\cite{Agmon}. 

\begin{lem} In any dimension $d\ge 3$,  
\EQ{\label{limap}
 \sup_{z\not\in\R} \|(\LH_0-z)^{-1}\fy\|_{L^{2,-1}} \lec \|\fy\|_{L^{2,1}}.} 
\end{lem}
As the proof below shows, we can even place the homogeneous weight $|x|^s$ near $x=0$ for $s<1/2$, but the above weaker estimate is sufficient for our purposes. 
\begin{proof}
Since $\s_{3} \D$ is diagonal, it suffices to consider scalar operators. 
Let $x=r\th$ be the polar coordinate, $\chi\in C_0^\I(\R)$ radial, $\chi(x)=1$ on $1/2<|x|<2$, and $\chi(x)=0$ on $|x|<1/3$ or $|x|>3$. Define operators $S_j$ for $j\in\Z$ by
\EQ{
 S_j\fy(x) = \chi(|x|)\fy(2^jx).}
After such a cut-off, we may regard the Besov norms as being defined on $\R$. 

First we derive \eqref{limap} from the following three estimates: for any $p>4$, 
\EQ{\label{polar Holder}
 \pt \sup_{j\in\Z}\|S_jr^{d/2-1}\fy\|_{L^2_\th C^{1/2}_r} \lec \|\fy_r\|_{L^2_x}, 
 \pq \sup_{j\in\Z}\|S_jr^{d/2-1/2}\fy\|_{L^2_\th (\dot B^{1/p}_{p,1})_r} \lec \|\fy\|_{\dot B^{1/2}_{2,1}},}
and for any $p<\I$, 
\EQ{ \label{sing int}
 \sup_{z\in\C\setminus\R}\Bigl |\int_0^\I\frac{f(r)}{\LR{r}-z}\, dr\Bigr| \lec \sum_{k=1}^2 \Bigl[\|r^{-k}f\|_{L^1_r}+\sup_{j\in\Z}\|r^{1-k}S_j f\|_{(\dot B^{1/p}_{p,1})_r}\Bigr] .}
Indeed, we can apply the latter estimate to 
\EQ{
 |\LR{(\D-z)^{-1}\fy|\psi}|\pt=\Bigl|\int_0^\I\int_S\frac{\hat{\fy}(r\th)\bar{\hat{\psi}}(r\th)}{\LR{r}-z}r^{d-1}\, d\th dr \Bigr|,}
by setting $f(r):=\int_S \hat{\fy}(r\th)\bar{\hat{\psi}}(r\th)r^{d-1}\, d\th$. The $L^1$ part is bounded by Cauchy-Schwarz and Hardy's inequalities: 
\EQ{ \| r^{-k} f\|_{L^{1}_{r}} \lec \| |\xi|^{-\frac{k}{2}} \hat\fy \|_{2}     \| |\xi|^{-\frac{k}{2}} \hat\psi \|_{2}
 \lec \|\hat\fy\|_{\dot H^{k/2}}\|\hat\psi\|_{\dot H^{k/2}}\lec \|\fy\|_{L^{2,k/2}}\|\psi\|_{L^{2,k/2}},}
 since $d\ge3$. Since $\dot B^{1/p}_{p,1}$ is an algebra, the Besov part is bounded by 
\EQ{
 \sum_{|j-h|\le 1}\|S_hr^{(d-k)/2}\hat\fy\|_{L^2_\th \dot B^{1/p}_{p,1} }\|S_hr^{(d-k)/2}\hat\psi\|_{L^2_\th \dot B^{1/p}_{p,1} },}
Therefore, using \eqref{polar Holder} and $C^{1/2}\subset\dot B^{1/p}_{p,1}$, we obtain 
\EQ{ \label{free resol est sharp}
 |\LR{(\D-z)^{-1}\fy|\psi}|  \lec \|\hat\fy\|_{\dot H^1\cap\dot B^{1/2}_{2,1}}\|\hat\psi\|_{\dot H^1\cap\dot B^{1/2}_{2,1}}
 \lec \|\fy\|_{L^{2,1}} \|\psi\|_{L^{2,1}},}
which implies \eqref{limap}. 

To prove \eqref{polar Holder}, one verifies that for  any $r_2>r_1>0$ 
\EQ{
 |[\fy(r\th)]_{r_1}^{r_2}|=\bigl|\int_{r_1}^{r_2}\fy_r(r\th)\, dr\bigr|
 \lec(1-r_1/r_2)^{1/2}r_1^{1-d/2}\|\fy_r(r\th)\|_{L^2_r},}
by Cauchy-Schwarz. Square integrating in $\th$, we obtain the first estimate 
\EQ{\label{eq:diff Best}
 \|S_j r^{d/2-1}\fy\|_{ L^2_\th C^{1/2}_r} \lec \|\fy_r\|_{L^2_x} \lec \|\fy\|_{\dot H^1}.}
The complex interpolation between \eqref{eq:diff Best} and the trivial estimate 
\EQ{ \label{triv est}
 \|S_j r^{d/2}\fy\|_{L^2_\th L^2_r} \lec \|\fy\|_{L^2_x},}
yields a uniform bound for $j\in\Z$ and $\al\in[0,1]$: 
\EQ{ \label{comp int}
 S_j r^{d/2-1+\al}: \dot H^{1-\al}_x \to L^2_\th(B^{(1-\al)/2}_{2/\al,2/\al})_r.}
Combining \eqref{triv est} with the Sobolev embedding, we have another bound 
\EQ{ \label{sob est}
 S_j r^{d/2}: L^2_x \to L^2_\th L^2_r \subset L^2_\th(B^{(\al-1)/2}_{2/\al,2})_r.}
Then the real interpolation between \eqref{comp int} and \eqref{sob est} gives us for $0<\be<1$, 
\EQ{
 S_j r^{d/2-(1-\al)\be}: (\dot B^{(1-\al)\be}_{2,1})_x \to L^2_\th(B^{(\be-1/2)(1-\al)}_{2/\al,1})_r.}
Choosing $0<\al<1/2$ and $\be(1-\al)=1/2$, we obtain  the second estimate of \eqref{polar Holder}. 

To prove \eqref{sing int}, let $z=\t+i\e$ with $\t,\e\in\R$ and $\e\not=0$. If $\t\le 1$, then on $r>0$, 
\EQ{
 |(\LR{r}-z)^{-1}| \le (\LR{r}-1)^{-1} = \LR{r}/r^2,}
 which gives rise to the $L^1$ part of \eqref{sing int}. If $\t>0$, then there is a unique $s>0$ such that $\LR{s}=\t$ and the integral has a singularity at $r=s$. But it is easy to see 
\EQ{
 \CAS{ r<s/2,\ 3s/2<r &\implies |(\LR{r}-z)^{-1}| \lec \LR{r}/r^2,\\
 s/2<r<3s/2 &\implies |(\LR{r}-z)^{-1}-(s\LR{s}^{-1}(r-s)+i\e)^{-1}| \lec \LR{s}/s^2.}}
Hence it suffices to bound 
\EQ{ \label{sing part}
 \int_{|\si|<s/2}\frac{f(s+\si)dr}{\si+i\e} = \int_{0}^{s/2}\left[\frac{f(s+\si)}{\si+i\e}+\frac{f(s-\si)}{-\si+i\e}\right]\, dt,}
where the integrand on the right equals 
\EQ{\label{eq:frac split}
 \frac{f(s+\si)-f(s-\si)}{\si+i\e} + \frac{-2i\e}{\si^2+\e^2}f(s-\si).}
To bound the first term, we employ an integral identity for any function $F(t)$
\EQ{
 \int_0^s\int_{\si/2}^\si F(t) dt\frac{2d\si}{(\si+i\e)^2} = \int_0^s\left[\frac{1}{t+i\e}-\frac{i\e}{(t+i\e)(2t+i\e)}\right]F(t) \, dt,}
which follows simply from Fubini. Hence \eqref{sing int} equals to 
\EQ{ \label{diff form}
 \int_0^{s/2}\int_{\si/2}^{\si}[f(s+t)-f(s+t-3\si/2)]\, dt \frac{2d\si}{(\si+i\e)^2} + R,}
where the remainder $R$ is bounded by
\EQ{
 |R|\lec \int_0^{s/2}\frac{\e}{\si^2+\e^2}[|f(s+\si)|+|f(s-\si)|]\, d\si
 \lec \|f\|_\I \lec \|f\|_{\dot B^{1/p}_{p,1}},}
while the leading term in \eqref{diff form} is bounded by using H\"older in $t$ and then by the difference norm of the Besov space 
\EQ{
 \lec\int_0^{s/2}\si^{-1/p}\|f(t)-f(t-3\si/2)\|_{L^p_t}\, \frac{d\si}{\si}
 \lec \|f\|_{\dot B^{1/p}_{p,1}}.}
This finishes the proof of \eqref{sing int}.
\end{proof}

Next, we transfer the estimate \eqref{limap} to the perturbed operator $\LH$.  

\begin{cor} \label{cor:limap}  
The operator  $\LH$ from~\eqref{eq:HHmatrix} satisfies 
\EQ{ \label{resol H}
 \sup_{z\not\in\s(\LH)}\|(\LH-z)^{-1}P_c\fy\|_{L^{2,-1}} \lec \|\fy\|_{L^{2,1}}}
\end{cor}
\begin{proof}
Since $Q^2\D^{-1}$ is compact $L^{2,-1}\to L^{2,1}$ and $P:=P_{c}=1-P_{d}$ is bounded on $L^{2,-1}$ and $L^{2,1}$, the compact perturbation argument implies that the resolvent estimate~\eqref{resol H}
can be violated only if $\LH$ exhibits a resonance. To be more specific, 
let $R_*(z)=(\LH_*-z)^{-1}$, $\LH=\LH_0+\K$ and $X=L^{2,-1} $. Suppose \eqref{resol H} fails, then there are $z_n\in\s(\LH)^c$ and $\fy_n\in L^2$ such that 
\EQ{ \label{contrad}
 \|\fy_n\|_{X^*}\to 0, \pq \psi_n:=R(z_n)P_c\fy_n, \pq \|\psi_n\|_{X}=1,\ \psi_n\rightharpoonup \psi\IN{\weak{X}}.}
Applying the resolvent expansion, we have 
\EQ{
 \psi_n=R_0(z_n)P\fy_n+R_0(z_n)\K \psi_n.}
Since $P:X^*\to X^*$ bounded and $R_0(z):X^*\to X$ uniformly bounded, the first term on the right is vanishing. For the second term, one has 
$\K P\psi_n\to \K P\psi$ in $X^*$, because $\K:X\to X^*$ compact. If $|z_n|\to \I$, then we infer from 
\EQ{
 R_0(z_n)=-z_n^{-1}(\LH_0-z_n-\LH_0)R_0(z_n)=z_n^{-1}(-1+\LH_0R_0(z_n)),}
that $R_0(z_n)f\to 0$ in $X$ for any $f\in\S\subset X^*$ dense. This with the uniform boundedness implies that $R_0(z_n)\to 0$ strongly as operator $X^*\to X$. Then $\psi_n\to 0$ strongly in $X$, contradicting \eqref{contrad}. Hence $z_n$ is bounded, and so extracting a subsequence, we may assume that $z_n\to\exists z\in\C$. Let $f_n:=\psi_n-R_0(z_n)P\fy_n$. Then $f_n\to\psi$ weakly in $X\subset\S'$, and so in $\S'$ we have 
\EQ{
 (\LH_0-z_n)f_n=\K\psi_n\to(\LH_0-z)\psi=\K\psi.}
Hence $\LH\psi=z\psi$ in the distribution sense, whereas $\psi=P\psi\in X$, since $P:X\to X$ continuous. If $\psi=0$, then $\K\psi_n\to 0$ strongly in $X^*$ and so does $\psi_n\to 0$ in $X$, contradicting \eqref{contrad}. Finally, using~\eqref{eq:split efunc}, one obtains 
\[
L_{+}\psi^{+} = z\psi^{-},\quad \psi^{-}=z\psi^{+}
\]
which implies that $\psi^{+}\ne0$ and $L_{+}\psi^{+}=z^{2}\psi^{+}$. 
If $\psi^+\in L^2$ then $\psi\in L^2$ is also an eigenfunction of $\LH$, which is a contradiction. Hence $\psi^+\in X\setminus L^2$. It now follows from the theory of scalar Schr\"odinger operators that the only possibility would be that $z^{2}=1$ is a threshold resonance of~$L_{+}$. However, in our case this is known not to occur, see~\cite{DS}, \cite{CHS}. 
\end{proof}

By the Kato theory~\cite{Kato}, \eqref{resol H} implies the weighted $L^2$ estimates on the evolution. 

\begin{lem}
For any $\e>0$, we have 
\EQ{ \label{wL2 H}
 \|e^{i\LH t}P_c\fy\|_{L^2_t L^{2,-1}_{x}} & \lec \|\fy\|_{L^2_x}, \\
  \Big\|\int_0^t e^{i\LH(t-s)}P_cf(s)\, ds \Big\|_{L^2_t L^{2,-1}_x} & \lec \|f\|_{L^2_s L^{2,1}_x.}}
\end{lem}
\begin{proof}
Define an operator $T$ by $T\fy:=\LR{x}^{-1}e^{it\LH}\fy$. Then 
\[
T^{*}f = \int_{-\I}^{\I}  e^{-is\LH}\LR{x}^{-1} f(s)\, ds. 
\]
Our goal is to show that $T:L^{2}_{x}\to L^{2}_{t,x}$ which is the same as $T T^{*}: L^{2}_{t,x}\to L^{2}_{t,x}$. 
Now 
\[
 (T T^{*} \, f)(t) = \int_{0}^{\I} \LR{x}^{-1}e^{i(t-s)\LH} \LR{x}^{-1}\, f(s)\, ds, \quad t>0
\]
where we first (and without loss of generality) consider positive times. Denoting the Fourier transform in time by $\F$, one checks that
\[
  (\F  T  T^{*} f)(\t) = \int_\R \LR{x}^{-1} R(\t-i0) R(\s+i0) \LR{x}^{-1}(\F f)(\s) \, d\s. 
\]
By the resolvent identity, the right-hand side is the same as 
\[
 \int_\R \LR{x}^{-1}\frac{R(\t-i0) - R(\s+i0)}{\t - \s - i0} \LR{x}^{-1}(\F f)(\s) \, d\s. 
\]
By Plancherel, the boundedness of the Hilbert transform on $L^{2}$, and Corollary~\ref{cor:limap}, we conclude
that $T T^{*}$ is bounded on $L^{2}_{t,x}$, which proves the first inequality in~\eqref{wL2 H}.

Similarly, the retarded estimate on $t>0$ follows by letting $\e\to+0$ in 
\EQ{\nonumber
 \Bigl\|\iint_{0<s<t} e^{i\LH(t-s)-\e(t-s)+it\t}Pf(s)\, dsdt\Bigr\|_{L^2_\t X} 
 & = \bigl\|\int_0^\I e^{is\t}R(\t-i\e)Pf(s)\, ds\bigr\|_{L^2_\t X} \\
 &\lec\bigl\|\int_0^\I e^{is\t}f(s)\, ds\bigr\|_{L^2_\t X^*}\simeq \|f\|_{L^2_{t>0}X^*}.}
The estimate on $t<0$ is obtained by reversing  the sign of $i\e$. 
\end{proof}

\subsection{Time-dependent, spatially-constant perturbation}
We now turn to the main goal of this section, i.e., proving the dispersive estimates
required in the proof of the center-stable manifold theorem. 
To be more specific, 
we consider a small time-dependent perturbation in the form 
\EQ{
 \ga_t = i\D\LL \ga + P_c[a(t)\cdot\na \ga + f], \pq \ga(0)=P\ga(0), \pq \|a\|_{L^\I_t}\ll 1,}
and follow Beceanu's approach~\cite{Bec1} in order to establish Strichartz estimates. 
By the time symmetry, we may restrict it to $t\ge 0$. 
Putting $u=(\ga,\bar{\ga})$ and $F=(f,\bar{f})$, it is rewritten in the matrix formulation 
\EQ{
 u_t = i\LH u + P_c[a(t)\cdot\na u + F], \pq u(0)=P_cu(0).}
Let $A(t)\fy:=a(t)\cdot\na\fy$. The equation can be extended to include the $P_d$ component 
\EQ{
 z_t \pt= i\LH P_c z - P_d z + A(t)P_cz + F, \pq z(0)=u(0),}
then $u:=P_cz$ satisfies the previous equation. The equation can be rewritten in the form 
\EQ{ \label{eq2 z}
 \pt z_t=[i\LH_0+V+A(t)]z+\ti F, 
 \pr V:=-i\LH_0P_d-P_d+i\K, \pq \ti F(t):=-A(t)P_d z(t)+F(t)}
 where $\LH=\LH_{0}+\K$. 
Let $U(t,s)$ be the evolution operator defined by the equation $u_t = A(t) u$, namely 
\EQ{
 U(t,s)\fy = \fy(x+b(t,s)), \pq b(t,s)=\int_s^t a(\t)d\t.}
Note that $\LH_0$ and $A(t)$ commute, so do their evolution operators. 
Regarding $\ti F$ as an input, we consider the Duhamel formula 
\EQ{ \label{eq3 z}
 \pt z = z_0+\int_0^t e^{(t-s)i\LH_0}U(t,s)[Vz(s)+\ti F(s)]\, ds, 
 \pq z_0:=e^{i\LH_0t}U(t,0)z(0).}
Now define space-time operators $T, T_{0}, T_{1}$ on $(t,x)\in[0,\I)\times\R^3$ by 
\EQ{
 \pt Tg = \int_0^t e^{(t-s)i\LH_0}U(t,s)V g(s)\, ds, 
 \pr T_0g = \int_0^t e^{(t-s)i\LH_0}Vg(s)\, ds, 
 \pq T_1g = \int_0^t e^{(t-s)(i\LH_0+V)}Vg(s)\, ds.}
Then we have 
\EQ{ \label{eq4 z}
 \pt (1-T)z = z_0 + \int_0^t e^{(t-s)i\LH_0}U(t,s)\ti F(s)\, ds,
 \pr (T_0-T)g = \int_0^t e^{(t-s)i\LH_0}(1-U(t,s))V g(s)\, ds}
as well as the following properties.

\begin{lem} \label{inv 1+T0}
$T_{0}, T_{1}$ are bounded on $L^2_t L^{2,-\s}_x$ for any $\s\ge 1$, where it holds 
$$(1-T_0) (1+T_1)=(1+T_1) (1-T_0)=1$$
\end{lem}
\begin{proof}
Let $i\ti\LH=i\LH_0+V$, then 
\EQ{
 T_0 T_1f\pt=\iint_{0<t_1<t_0<t} e^{(t-t_0)i\LH_0}V e^{(t_0-t_1)i\ti\LH}V f(t_1)dt_1dt_0
 \pr=\int_0^t \int_0^{t-t_1}e^{(t-t_1-s)i\LH_0}Ve^{si\ti\LH}Vf(t_1)ds dt_1
 \pr=\int_0^t [e^{(t-t_1)i\ti\LH}-e^{(t-t_1)i\LH_0}]Vf(t_1)dt_1 = T_1f-T_0f,}
where in the third step we used the Duhamel formula
\EQ{
 e^{it\ti\LH}\fy = e^{it\LH_0}\fy + \int_0^t e^{(t-s)i\LH_0}Ve^{is\ti\LH}\fy ds.}
Since $i\LH_0=i\ti\LH- V$, we obtain $T_1 T_0=-T_0+T_1$ in the same way. 
The weighted $L^2$ estimate \eqref{wL2 H} for $\LH$ implies that if $\s\ge 1$ then 
\EQ{
 \|T_1f\|_{L^2_t L^{2,-\s}_x} \pt \le
 \|\int_0^t[e^{i\LH(t-s)}P_c Vf(s)+e^{-(t-s)} e^{-i(t-s)\LH_{0}}  P_d Vf(s)]\, ds\|_{L^2_t {L^{2,-1}}}
 \pr\lec \|P_cVf\|_{L^2_tL^{2,1}}+\|P_d Vf\|_{L^2_t {L^{2,-1}}}\lec \|f\|_{L^2_tL^{2,-\s}_x},}
since $V:L^{2,-\s} \to L^{2,1}$ bounded. The bound for $T_0$ is simpler, as it uses the free analogue of~\eqref{wL2 H}.
\end{proof}

Next,  we extend the weighted $L^2$ estimate to $T$ with some small loss of regularity. For that purpose, we introduce the bilinear form in dyadic frequency 
\EQ{
 I_j^M(f,g):=\int_0^\I\int_0^{\max(t-M,0)}\LR{e^{(t-s)i\D}U(t,s)\De_j wf(s)|wg(t)}\, dsdt,}
for $j\ge 0$, $M\ge 1$, $w(x):=\LR{x}^{-\s}$, $f(t,x),g(t,x)\in L^2_{t,x}$ and
 $1=\sum_{j=0}^\I\De_j$ is a Littlewood-Paley decomposition defined by $\De_j u = \F^{-1}\fy(2^{-j}\x)\hat u(\x)$ for all $j\ge 1$, 
 with some radial non-negative $\fy\in C_0^\I(\R^3)$ supported on $1/2<|\x|<2$. 

To sum $I_j$ over $j\ge 0$, we use the almost orthogonality 
\EQ{ \label{alm ort}
 \pt I_j^M(f,g)= I_j^M(Q_jf,Q_jg),
 \pq Q_jf=\sum_{|k-j|\le 1}\De_k f+w^{-1}[\De_k,w]f,}
where the commutator term is small in the sense 
\EQ{
 \|w^{-1}[\De_k,w]f\|_2 \lec 
 2^{-k}\|f\|_2.}

For each $I_j^M$, we now have the following estimate.
\begin{lem} \label{est IjM}
Let $\s>3/2$ and  $M\ge 1$. Then for any $j\ge 0$, 
\EQ{
 |I_j^M(f,g)| \le C(\s)\,  2^{2j} M^{\frac32-\s}   \|f\|_{L^2_{t,x}}\|g\|_{L^2_{t,x}},}
provided that $\|a\|_{L^\I_t}\ll 1$. 
\end{lem}
\begin{proof}
Define space-time functions $K_j$ and $K^j$ by 
\EQ{
 K_j(t,x)= 2^{3j}K^j(2^jt,2^jx) = e^{it\D}2^{3j}\hat\fy(2^jx).}
Then we can rewrite $I_j^M$ as follows
\begin{align}  \label{Ij int}
 I_j^M(f,g)  &= J_j^M(f,g)+\bar{J_j^M(g,f)} \\
  J_j^M(f,g) &= \int_{\substack{|y|>|x|, \\ s>0,\ t>s+M,}} K_j(t-s,x-y+b)w(y)f(s,y)w(x)\bar{g}(t,x)\, dsdtdxdy \nonumber 
 \\& =\int_{\substack{|y-b|>|y+z|, \\ s>0,\ u>M,}}2^{3j}K^j(2^ju,2^jz)w(y-b)f(s,y-b)
 \nonumber \prQQ\qquad \times w(y+z)\bar{g}(s+u,y+z)\, dsdudydz,\nonumber
 \end{align}
where $b=b(t,s)=\int_s^ta(\sigma)\, d\si$. Now we use the pointwise decay of $K^j$: for $t\gtrsim1$, 
\EQ{ \label{dec Kj}
 |K^j(t,x)| \lec \CAS{t^{-1}\LR{|t-|x||+2^{-2j}t}^{-N} &(|t-|x||\not\simeq 2^{-2j}t), \\
  t^{-1}\LR{t-|x|}^{-1/2} &(|t-|x||\simeq 2^{-2j}t),}}
where $N\in\N$ is arbitrary. Note that the first line is similar to the wave equation, while the second term is the stationary phase part due to the mass term.
The proof of~\eqref{dec Kj} will be given below.  In particular, \eqref{dec Kj} implies the uniform estimate
\EQ{
\label{dec Kj univ}
|K^{j}(t,x)|\lec 2^{j}t^{-\frac32}
}
whence 
\EQ{\label{easy Jbd}
|J_{j}^{M}(f,g)| &\lec 2^{\frac{5j}{2}} \int_{u>M} u^{-\frac32} \| w(x)f(s,x)\|_{L^{1}_{x}} \| w(y)g(s+u,y)\|_{L^{1}_{y}}\, ds du\\
&\lec 2^{\frac{5j}{2}} M^{-\frac12} \|f\|_{L^{2}_{t,x}} \|g\|_{L^{2}_{t,x}} 
}
This bound is not too useful for our purposes since we require more decay in~$M$. 
Using \eqref{dec Kj} and arguing as in~\eqref{easy Jbd}, the integral \eqref{Ij int} on $u\not\simeq|z|$ is bounded by
\EQ{
 \pt\int_{u>M} 2^{3j}(2^ju)^{-3-N}|w(y-b)f(s,y-b)w(y+z)g(s+u,y+z)|\, dzdydu ds\\
 &\lec 2^{-jN} M^{-2-N}\|f\|_{L^2_{t,x}}\|g\|_{L^2_{t,x}}.}
If $u\simeq|z|$, then  
\EQ{
 |b(t,s)| \ll |t-s|= u  \simeq |z| \le |y-b|+|y+z|+|b| \lec |y-b|}
 since $|y+z|<|y-b|$. 
The integral \eqref{Ij int} over this region is therefore bounded by, with $b=b(s,u)$,  
\EQ{\nonumber
 \pt\int_{\substack{u>M \\   |u|\lec |y-b|   }}   \frac{2^{3j}}{2^ju}  \frac{|f(s,y-b)g(s+u,y+z)|}{\LR{y-b}^{\s}\LR{y+z}^{\s}}  \, dydzduds
 \pr\lec    \int_{ u>M  }   2^{2j} u^{-1}  \|\LR{y-b}^{-\s}\LR{y+z}^{-\s}\|_{L^2_{[y:|y-b|\gec u]}L^2_z  }  \|f(s,\cdot)\|_{L^2_x}    \|g(s+u,\cdot)\|_{L^2_x}\, dsdu
 \pr \lec \int_{ u>M  }   2^{2j} u^{\frac12-\sigma}  \|f(s,\cdot)\|_{L^2_x}    \|g(s+u,\cdot)\|_{L^2_x}\, dsdu 
 \pr\lec 2^{2j} M^{\frac32-\s}    \|f\|_{L^2_{t,x}}   \|g\|_{L^2_{t,x}}}
 and the lemma is proved.  
\end{proof}

\begin{proof}[Proof of \eqref{dec Kj}]
From the explicit form of $\hat{\sigma_{S^{2}}}$, one has 
  (up to some nonzero multiplicative constants) 
\EQ{ \label{form Kj}
 K^j(t,r) = \int_{-1}^1 e^{it\LR{\ro}_j+icr\ro}\psi(\ro)\, d\ro dc
 \pn = r^{-1}\int e^{it\LR{\ro}_j}\sin(r\ro)\psi(\ro)\, d\ro,}
where $\LR{\ro}_j:=\sqrt{\ro^2+2^{-2j}}$ and $\psi\in C_0^\I(1/2,2)$. If $r<t/2$, then we use the first form. Let $\Phi_c(\ro):=t\LR{\ro}_j+cr\ro$, then 
\EQ{
 \Phi_c'=t\frac{\ro}{\LR{\ro}_j}+cr=O(t), \pq \Phi''_c=t\frac{2^{-2j}}{\LR{\ro}_j^3}=O(t2^{-2j})\gec|\Phi_c^{(1+k)}|,
}
for all $k\ge 0$. Integration by parts in $N$ times gives 
\EQ{
 \int e^{i\Phi_c}\psi(\ro)\, d\ro=\int e^{i\Phi_c}[i\p_\ro(\Phi_c')^{-1}]^N\psi(\ro)\,  d\ro,}
where the right-hand side is bounded by $t^{-N}\|\psi\|_{W^{N,\I}}$. 
It remains to estimate in the region $r>t/2$ and $|t-r|\gec 1$, for which we use the second form in \eqref{form Kj}. It can have stationary phase only at $\ro=\ro_0$ solving
\EQ{
 0 = \Phi_{-1}'(\ro_0) = t-r-\frac{2^{-2j}t}{\LR{\ro_0}_j(\LR{\ro_0}_j+\ro_0)},}
which has at most one solution in $\supp\psi$ because $\Phi_c''>0$. If $|t-r|\not\simeq 2^{-2j}t$, then it has no solution in $\supp\psi$ and 
\EQ{
 |\Phi_{-1}'| \simeq |t-r|+2^{-2j}t \gec |\Phi_{-1}^{(1+k)}|,}
for all $k\ge 0$. Hence by the same argument of non-stationary phase, we obtain the bound $t^{-1}(|t-r|+2^{-2j}t)^{-N}$ for any $N\ge 0$. Finally, if $|t-r|\simeq 2^{-2j}t$, then we integrate only once and in the region $|\ro-\ro_0|>\de\in(0,1)$, where $\Phi_{-1}''\simeq |t-r|$ and $\Phi_{-1}'\simeq |t-r|(\ro-\ro_0)$, so 
\EQ{
 |K^j(t,r)| \lec \de + \int_\de^1 \frac{|t-r|}{(|t-r|\ro)^2}d\ro \lec \de + (|t-r|\de)^{-1}.}
Choosing $\de=|t-r|^{-1/2}$, we obtain the claimed estimate in this region. 
\end{proof}

We can now state a key dispersive estimate. 

\begin{lem}
Let $1>\nu>0$ and $\s\ge \frac12+\frac{2}{\nu}$. If $\|a\|_{L^\I_t}\ll 1$, then  
\EQ{
 \|e^{it\LH_0}U(t,0)\fy\|_{L^2_tL^{2,-\s}_x} &\lec \|\fy\|_{H^{\nu/2}_x}, \\
\Big \|\int_0^t e^{i(t-s)\LH_0}U(t,s)f(s)\, ds \Big\|_{L^2_tL^{2,-\s}_x} &\lec \|\D^\nu f\|_{L^2_t L^{2,\s}_x+L^1_tH^{-\nu/2}_x}.}
\end{lem}
\begin{proof}
The short-time interaction part is estimated simply by the $L^2_x$ conservation and integration in $t$: 
\EQ{
 |I_j^M(f,g)-I_j^0(f,g)| \lec M\|f\|_{L^2_{t,x}}\|g\|_{L^2_{t,x}},}
Hence for any $1>\nu>0$, we choose $M=2^{\nu j}$ and  $\s$ as above, whence 
\EQ{
2^{2j} M^{\frac32-\s} \le M=2^{j\nu} }
Thus Lemma~\ref{est IjM} implies that 
\EQ{
 |I_j^0(f,g)| \lec 2^{\nu j}\|f\|_{L^2_{t,x}}\|g\|_{L^2_{t,x}}.}
Using \eqref{alm ort}, we obtain 
\EQ{
 \pt \Big |\iint_0^t\LR{e^{i(t-s)\LH_0}U(t,s)wf(s)|wg(t)}\, dsdt \Big|\le \sum_{j=0}^\I|I_j^0(f,g)| 
 \pr\lec \sum_{j=0}^\I\sum_{|j-k|+|j-\ell |\le 2} 2^{\nu j}\|\De_kf\|_{L^2_{t,x}}\|\De_\ell g\|_{L^2_{t,x}}+2^{(\nu-1)j}\|f\|_{L^2_{t,x}}\|g\|_{L^2_{t,x}}
 \pr\lec \|f\|_{L^2_tH^{\nu}_x}\|g\|_{L^2_t L^2_x},}
which implies through duality 
\EQ{
 \pt \Big\|\int_0^t e^{i(t-s)\LH_0}U(t,s)f(s)\, ds \Big\|_{L^2_t L^{2,-\s}_x}
 \pn\lec \|\D^{\nu}\LR{x}^\s f\|_{L^2_{t,x}} \simeq \|\D^\nu f\|_{L^2_t L^{2,\s}_x},}
and the same estimate for the integral on $t<s<\I$. The remaining estimates follow by the usual duality argument. 
\end{proof}

\begin{lem}
For $\s>14$, one has 
\EQ{
 \|(T_0-T)g\|_{L^2_tL^{2,-\s}_x} \lec \|a\|_{L^\I_t}^{\frac14}\|g\|_{L^2_t L^{2,-\s}_x}.}
Moreover, if $\|a\|_{L^\I_t}$ is small enough, then $1-T$ is invertible on $L^2_tL^{2,-\s}_x$
as well as on $L^2_t \D^{\nu/2} L^{2,-\s}_x$ for $\nu>0$ small. 
\end{lem}
\begin{proof}
Decompose $T_0-T$ into the long and short time interactions as before: 
\EQ{
 \pt (T_0-T) g = T_S g + T_L g, 
 \pr T_S g:= \sum_{j=0}^{\I} T_{S}^{j}= \sum_{j=0}^{\I}  \int_{\max(t-M_{j},0)}^t e^{(t-s)i\LH_0}(1-U(t,s))\Delta_{j} V g(s)\, ds,}
where $M_{j}\gg 1$ will be chosen later and $\Delta_{j}$ is as above. $T_{L}^{j}$ is defined in a similar fashion.
We estimate $T_S^{j}$ in two ways: on the one hand,  by definition of $U(t,s)$,  and the fact that $V$ gains a derivative, 
\EQ{
 \pt\|e^{(t-s)i\LH_0}(1-U(t,s))\Delta_{j} V \fy\|_{L^2}
 \pn\le\|(1-U(t,s))V\fy\|_{L^2}
 \pr\le |b(t,s)|\|\na V\fy\|_{L^2} \le |t-s|\|a\|_{L^\I}\|\fy\|_{L^{2,-\s}},}
 for any $\s$, and on the other hand,
 \EQ{
 \pt\|e^{(t-s)i\LH_0}(1-U(t,s))\Delta_{j} V  \fy\|_{L^2} 
 \pn \le 2 \||\nabla|^{-1}\Delta_{j} \nabla V\fy\|_{2} 
  \lec 2^{-j} \| \fy \|_{L^{2,-\s}}
 }
 Hence, 
 \EQ{
  \pt\|e^{(t-s)i\LH_0}(1-U(t,s))\Delta_{j} V \fy\|_{L^2}
   \lec  |t-s|^{\frac12} 2^{-\frac{j}{2}} \|a\|_{L^\I}^{\frac12}  \|\fy\|_{L^{2,-\s}}
 }
By integration in $s$,   for any $\s\in\R$, 
\EQ{
 \|T_S^{j}g\|_{L^2_t L^2_x} \lec M_{j}^{\frac32} 2^{-\frac{j}{2}}\|a\|_{L^\I_t}^{\frac12}  \|g\|_{L^2_t L^{2,-\s}_x}.}
For $T_L^{j}$ we use Lemma \ref{est IjM}, with and without $a$ (or $a=0$), and $\s$ large. Taking account of the derivative gain from $V$ and its decay, we obtain 
\EQ{
 \|T_L^{j} g\|_{L^2_t L^{2,-\s}_x} \lec 2^{2j} M_{j}^{\frac32-\s}   \|g\|_{L^2_t L^{2,-\s}_x}.}
Hence, by choosing $M_{j}=2^{\frac{j}{6}}\|a\|_{L^\I}^{-\frac16}$, $\s>14$, and summing in $j\ge0$, we obtain 
\EQ{
 \|(T_S+T_L)g\|_{L^2_t L^{2,-\s}_x} \lec    \|a\|_{L^\I_t}^{\frac14}    \|g\|_{L^2_t L^{2,-\s}_x}.}
Since $1-T_0$ is invertible on $L^2_t L^{2,-\s}_x$ by Lemma \ref{inv 1+T0}, so is $1-T$ by the Neumann series, if $\|a\|_{L^\I_t}$ is small enough. 
\end{proof}

Now fix $\nu>0$ small and $\s$ large. Going back to the equation \eqref{eq4 z} of $z$, we obtain via the previous lemmas, 
\EQ{
 \|\D^{-\nu/2}z\|_{L^2_t L^{2,-\s}_x} \pt\lec \|\D^{-\nu/2}z_0\|_{L^2_t L^{2,-\s}_x}+\|\D^{\nu/2}\ti F\|_{L^2_t L^{2,\s}_x+L^1_tH^{-\nu/2}_x}
 \pr\lec \|z(0)\|_2+ \|a\|_{L^\I_t}\|\D^{-\nu/2}z\|_{L^2_t L^{2,-\s}_x} + \|F\|_{L^2_t \D^{-\nu/2} L^{2,\s}_x +L^1_tL^2_x},}
and so, for $\|a\|_{L^\I_t}$ small enough, 
\EQ{
 \|z\|_{L^2_t\D^{\nu/2} L^{2,-\s}_x} \lec \|z(0)\|_2  + \|F\|_{L^2_t \D^{-\nu/2} L^{2,\s}_x +L^1_tL^2_x}.}
On the other hand, applying the Strichartz estimate for $e^{i\LH_0t}$ (see, for example, \cite[Section 4]{ScatBlow}) to \eqref{eq3 z}, 
yields 
\EQ{
 \|z\|_{Stz} \pt\lec \|z_0\|_{Stz}+\|Vz+\ti F\|_{Stz^*}
 \pn\lec \|z(0)\|_2+\|z\|_{L^2_t \D^{\nu/2} L^{2,-\s}_x}+\|F\|_{Stz^*},}
where the Strichartz space is defined by 
\EQ{
 Stz:=L^\I_tL^2_x\cap L^2_tB^{-5/6}_{6,2}.} 
Thus, we obtain the following estimate as a conclusion to this section: 
\EQ{
 \|z\|_{Stz} \lec \|z(0)\|_2 + \|F\|_{L^1_tL^2_x+L^2_t(B^{5/6}_{6/5,2}\cap\D^{-\nu/2} L^{2,\s})_x}.}
As far as  scattering  is concerned, we note that 
\EQ{
 \|z-z_\I\|_2 \to 0 \pq (t\to\I),
 \pq z_\I = z_0+\int_0^\I e^{i(t-s)\LH_0}U(t,s)\ti F(s)\, ds.}
Hence there is $\fy_+\in L^2$ such that 
\EQ{
 \|U(0,t)z(t)-e^{i\LH_0t}\fy_+  \|_2 \to 0.}
The following result summarizes the dispersive estimates of this section. 

\begin{prop} \label{prop:Stz ga}
Let $\nu>0$ be small and $\s$ large.  Then there exists a small $\de>0$ such that if $\|a\|_{L^\I_t}\le\de$, then the solution of 
\EQ{
 \ga_t = i\D\LL \ga + P_c[a(t)\cdot\na \ga + f], \pq \ga(0)=P_{c}\ga(0),}
satisfies on any interval $I\ni 0$, 
\EQ{
 \pt\|\ga\|_{L^\I_tL^2_x \cap L^2_t(B^{-5/6}_{6,2}\cap \D^{\nu/2}L^{2,-\s})_x} \lec \|\ga(0)\|_2 + \|f\|_{L^1_tL^2_x+L^2_t(B^{5/6}_{6/5,2}\cap\D^{-\nu/2} L^{2,\s})_x}, 
}
and for some $\ga_+\in L^2$, and $b(t,0)=\int_0^t a(s)\, ds$, 
\EQ{
 \|\ga(t,x-b(t,0))-e^{i\D t}\ga_+\|_2 \to 0 \pq (t\to\I).}
\end{prop}

\end{document}